\documentclass{siamart251216}

\usepackage[T1]{fontenc}
\usepackage[utf8]{inputenc}
\usepackage{microtype}
\usepackage{amsmath,amssymb,mathtools,mathrsfs}
\usepackage{stmaryrd}
\usepackage{booktabs}
\usepackage{graphicx}
\usepackage{float}
\usepackage{placeins}
\usepackage{needspace}
\makeatletter
\@ifundefined{algorithm}{\usepackage{algorithm}}{}
\makeatother
\usepackage{algpseudocode}
\algrenewcommand\algorithmicrequire{\textbf{Input:}}
\algrenewcommand\algorithmicensure{\textbf{Output:}}

\newsiamthm{assumption}{Assumption}
\newsiamremark{remark}{Remark}

\crefname{assumption}{Assumption}{Assumptions}
\crefname{algorithm}{Algorithm}{Algorithms}

\newcommand{\Qir}{\mathcal Q_{\infty}}
\newcommand{\Qirs}{\mathcal Q_{\infty}^*}
\newcommand{\Qlr}{\mathcal Q_{\ell}}
\newcommand{\Qlrs}{\mathcal Q_{\ell}^*}
\newcommand{\norm}[1]{\left\lVert #1\right\rVert}
\newcommand{\ii}{\mathrm i}
\newcommand{\supp}{\operatorname{supp}}

\newcommand{\TheTitle}{An adaptive localized orthogonal decomposition method
for Helmholtz problems}

\headers{Adaptive localized orthogonal decomposition}{Y. Li and C. Que}
\title{\TheTitle}
\author{
Yuwen Li\thanks{School of Mathematical Sciences, Zhejiang University, 866 Yuhangtang Road, Hangzhou 310058, People's Republic of China
(\email{liyuwen@zju.edu.cn}, \email{qcx@zju.edu.cn}).
\funding{This work was supported by the National Key R\&D Program of China
under grant 2024YFA1012600.}}
\and
Chuxin Que\footnotemark[1]
}

\begin{document}

\maketitle

\begin{abstract}
We develop an adaptive localized orthogonal decomposition (ALOD) method
for Helmholtz problems with parametrized loads. Under standard resolution,
oversampling, and discrete-stability conditions, we establish, to our
knowledge, the first wavenumber-robust residual a posteriori error estimate
for a Helmholtz LOD discretization. The estimate is two-sided relative to
the fine-grid solution and includes a computable bound for the
corrector-localization error. In addition, we design a novel family-aware adaptive algorithm using additive Schwarz corrections to construct a
shared regional enrichment space for Helmholtz problems with multiple loads. Energy-weighted proper orthogonal decomposition compresses this space. We establish stability and quasi-optimality of the enriched formulation, with bounds independent of the enrichment space. Numerical experiments demonstrate the effectiveness of ALOD for multiple loads and large wavenumbers.
\end{abstract}


\section{Introduction}\label{sec:introduction}

Helmholtz problems arise in the numerical simulation of time-harmonic wave
propagation, including applications in acoustics, electromagnetics, and
seismic imaging. In inverse problems, uncertainty quantification, and design,
these applications often require repeated solutions for many sources or
incident fields.
Helmholtz problems impose two distinct resolution demands: wave propagation
is global, whereas corners and localized sources require fine resolution only
in selected regions. For standard finite elements, resolving the wavelength
does not by itself prevent pollution at large wavenumbers
\cite{BabuskaSauter1997,MelenkSauter2011,EsterhazyMelenk2012}.
This pollution reflects accumulated phase error and can remain substantial
even when the wavelength is nominally resolved. Suppressing it therefore
requires increasingly stringent global resolution as the wavenumber grows,
making repeated high-frequency simulations expensive.
Wavenumber-explicit a priori analyses quantify this effect for linear and
higher-order methods, including sharp relative-error estimates for
nontrapping and variable-coefficient scattering problems
\cite{Wu2014,DuWu2015,LafontaineSpenceWunsch2022,GalkowskiSpence2025}.

Localized orthogonal decomposition (LOD) separates these roles: local
fine-scale correctors encode unresolved oscillations, while an
operator-adapted coarse space remains online
\cite{MalqvistPeterseim2014,HenningMalqvist2014,MalqvistPeterseim2020}. For the Helmholtz equation,
the a priori theory proves stability and a first-order energy-error estimate
with constants independent of the wavenumber, provided that the coarse mesh
resolves the wavelength, the oversampling depth grows logarithmically with the
wavenumber, and the local corrector problems are resolved accurately
\cite{GallistlPeterseim2015,Peterseim2017}. In this regime, LOD eliminates pollution while retaining
an online space at the natural wavelength scale. Multiresolution LOD further
decouples the discretization scales and provides rigorous stability and a
priori error control, while super-localized LOD improves the decay of the
localization error and relaxes the required oversampling
\cite{HauckPeterseim2022,FreeseHauckPeterseim2024}. Parallel patch solvers and
reduced corrector models address the remaining setup cost
\cite{EngwerHenningMalqvistPeterseim2019,AbdulleHenning2015,KeilRave2023}. These results prescribe
resolution and localization parameters a priori. They do not provide
computable indicators that distinguish insufficient coarse resolution,
fine-scale resolution, and localization depth on an adaptive mesh hierarchy.

Standard adaptive finite element methods (AFEMs) for Helmholtz equations control the discretization error of
conforming, stabilized, and hybridizable discontinuous Galerkin finite element
solutions through residual or equilibrated-flux constructions
\cite{DuanWu2023,ChaumontFreletErnVohralik2021,CamargoRojasVega2025}.
They do not control the additional localization error caused by truncated LOD
correctors or distinguish it from coarse- and fine-scale discretization
errors. This paper closes that gap. To the best of our knowledge, it gives the first
wavenumber-robust residual a posteriori error estimate for an LOD
discretization of the Helmholtz equation, within the standard LOD resolution,
oversampling, and discrete-stability regime. The estimate is two-sided
relative to the current fine-grid solution and includes a computable bound for
the corrector-localization error. It therefore supports separate updates of
the online coarse space, the fine mesh used for correctors and residuals, and
the oversampling level. %Its geometric constants are independent of fine-mesh refinement depth, while the required stability dependence remains explicit.

The load-family extension uses the same indicator. Such families
arise in inverse scattering with many incident fields, source localization
with many candidate sources, seismic or ultrasound imaging with multiple
transmitters, and uncertainty quantification or design with repeated forward
solves. Repeating the adaptive construction for every right-hand side is
costly, but a space adapted to one reference load may miss features activated by
other loads. We therefore
reuse the local residual Riesz representatives of the training family as
regional additive Schwarz corrections and combine them into a shared
enrichment. Proper orthogonal decomposition and reduced-basis ideas enter only
to compress this residual-generated space, not to replace LOD by a separate
global surrogate. This connects problem-adapted enrichment
\cite{MelenkBabuska1996,BabuskaMelenk1997}, reduced spaces
\cite{KunischVolkwein2002,PrudhommeEtAl2002,RozzaHuynhPatera2008,
BinevEtAl2011}, and localized adaptive reduction
\cite{OhlbergerSchindler2015,BuhrEngwerOhlbergerRave2017}. %Kernel-lifted tests provide enrichment-independent stability and quasi-optimality, while energy-weighted POD and per-load checks control compression.

The rest of the paper is organized as follows.
Section~\ref{sec:model-spaces} introduces the Helmholtz problem and the LOD
construction.
Section~\ref{sec:certificates} develops residual and corrector-localization
estimates and the resulting fine-grid error bound.
Section~\ref{sec:family-enrichment} develops family-aware adaptivity through
shared regional enrichment and coarse--fine refinement.
Section~\ref{sec:numerics} reports the numerical experiments.
% Section~\ref{sec:conclusions} summarizes the conclusions and outlook.

\section{Problem setting and LOD spaces}\label{sec:model-spaces}

\subsection{Helmholtz problem and mesh hierarchy}\label{sec:model-problem}
Let $\Omega\subset\mathbb R^2$ be a bounded polygonal Lipschitz domain.
Its boundary is partitioned into relatively open parts $\Gamma_D$,
$\Gamma_N$, and $\Gamma_R$, disjoint up to endpoints, with
$|\Gamma_R|>0$. For a wavenumber $k\geq k_{\min}>0$, we consider the Helmholtz problem
\begin{equation*}
\left\{\begin{aligned}
 -\Delta u-k^2u&=f &&\text{in }\Omega,\\
 u&=0 &&\text{on }\Gamma_D,\\
 \partial_\nu u&=0 &&\text{on }\Gamma_N,\\
 \partial_\nu u-\ii k u&=0 &&\text{on }\Gamma_R.
\end{aligned}\right.
\end{equation*}
where $\nu$ denotes the outward unit normal. We work with complex-valued
function spaces and use $L^2$ inner products that are linear in the first
argument and conjugate linear in the second. On the space
$V=\{v\in H^1(\Omega;\mathbb C):v|_{\Gamma_D}=0\}$, we define the
energy inner product, its norm, and the Helmholtz form by
\begin{align}
 b_k(v,w)&=(\nabla v,\nabla w)_\Omega+k^2(v,w)_\Omega,
 &\norm v_k^2&=b_k(v,v),\label{eq:energy}\\
 a(v,w)&=(\nabla v,\nabla w)_\Omega-k^2(v,w)_\Omega
          -\ii k(v,w)_{\Gamma_R}.\label{eq:form}
\end{align}
For $D\subseteq\Omega$, $\norm v_{k,D}$ is the corresponding local energy
norm. The continuity bound
$|a(v,w)|\leq C_a\norm v_k\norm w_k$ holds with $C_a$ depending on
$\Omega$ and $k_{\min}$, but not on the meshes or on $k\geq k_{\min}$.
For a subspace $X\subseteq V$, let $X^\star$ be its continuous anti-dual space,
with norm
\[
 \norm G_{X^\star}=\sup_{0\neq w\in X}\frac{|G(w)|}{\norm w_k}.
\]
Given $F\in V^\star$, we consider the weak problem: find $u\in V$ such that 
\[a(u,v)=F(v),\qquad v\in V,\] 
where $F(v)=(f,v)_\Omega$ for the volume load in the strong formulation.
We are interested in repeated solves with the same form $a$ for a load
family $\{F_\mu\}_{\mu\in\mathcal P}\subset V^\star$. For each
$\mu\in\mathcal P$, let $u_\mu\in V$ denote the corresponding exact
solution. To construct a space for this family, we base adaptation on a finite training set
$\mathcal P_{\rm tr}=\{\mu_1,\ldots,\mu_M\}$.

At adaptive state $n$, we use conforming triangular meshes
$\mathcal T_{H,n}\preceq\mathcal T_{h,n}$, where $\preceq$ denotes
refinement. The associated continuous piecewise affine spaces satisfy
$V_{H,n}\subseteq V_{h,n}\subset V$. We assume that both mesh sequences
are nested and uniformly shape regular, and that they resolve the
boundary partition. The two spaces have distinct roles. The coarse
space determines the base online dimension, whereas the fine space resolves
the corrector and residual problems. When a mesh pair is fixed, we omit
$n$ and write $H_T=\operatorname{diam}T$ and $H|_T=H_T$. We similarly
use $h_K$ and $h_e$ for element and edge diameters on a fine or auxiliary
mesh. Throughout the paper, $\lesssim$ suppresses constants with the
stated dependencies.

On the current fine mesh, we consider the Galerkin discretization: find $u_h\in V_h$ such that 
\[a(u_h,v_h)=F(v_h),\qquad v_h\in V_h.\] 
For the load $F_\mu$, we denote this solution by $u_{h,\mu}$.
We measure the stability of
this problem by the inf--sup constant
\begin{equation}\label{eq:gamma-fine}
 \gamma_h(k)=\inf_{0\neq v_h\in V_h}\sup_{0\neq w_h\in V_h}
 \frac{|a(v_h,w_h)|}{\norm{v_h}_k\norm{w_h}_k}.
\end{equation}
We assume that the current fine-grid Galerkin problem is uniquely
solvable, i.e.,
\[
  \gamma_h(k)>0.
\]
For fixed $k$, this condition holds once the fine space is sufficiently
rich under the standard approximation hypotheses for compactly perturbed
elliptic problems \cite{BespalovHaberlPraetorius2017}.

\subsection{Quasi-interpolation and the fine-scale kernel}
\label{sec:interpolation}
We first introduce the coarse patches used for interpolation and
localization. For a coarse entity or a union of entities $S$, define
\[
 N^0(S)=S,\qquad
 N(S)=\bigcup\{T\in\mathcal T_H:\overline T\cap\overline S\neq\varnothing\},
 \qquad N^m(S)=N(N^{m-1}(S)).
\]
We denote the maximum patch cardinality by
$C_{\rm patch}(m)=\max_{T\in\mathcal T_H}\#\{K\in\mathcal T_H:
K\subseteq N^m(T)\}$.
Since mesh adjacency is symmetric, these patches satisfy
\begin{equation}\label{eq:patch-overlap}
 \sum_{T\in\mathcal T_H}\norm v_{L^2(N^m(T))}^2
 \leq C_{\rm patch}(m)\norm v_{L^2(\Omega)}^2.
\end{equation}
Shape regularity bounds $C_{\rm patch}(m)$ for each fixed $m$, but does not
imply polynomial growth in $m$ on arbitrarily graded meshes. However, we emphasize that no such
growth assumption is needed for the a posteriori estimates below.

Let $\mathcal N_H$ be all coarse vertices and
$\mathcal N_{H,D}=\mathcal N_H\cap\overline{\Gamma_D}$.
We use the projection $I_H=E_H\Pi_H^{\rm dg}:V\to V_H$, where
$\Pi_H^{\rm dg}$ is the elementwise $L^2$ projection onto discontinuous
piecewise affine functions and
\[
 (E_Hq)(z)=
 \begin{cases}
 0,&z\in\mathcal N_{H,D},\\
 \displaystyle\frac{\sum_{T\ni z}q|_T(z)|T|}{\sum_{T\ni z}|T|},
 &z\notin\mathcal N_{H,D}.
 \end{cases}
\]
This construction is a projection, so $I_Hv_H=v_H$, and it enlarges
supports by at most one coarse layer. We will use its standard local
approximation and stability estimates,
\begin{align}
 H_T^{-1}\norm{v-I_Hv}_{L^2(T)}
 &\leq C_{\rm app}\norm{\nabla v}_{L^2(N(T))},\label{eq:IH-approx}\\
 \norm{\nabla I_Hv}_{L^2(T)}
 &\lesssim\norm{\nabla v}_{L^2(N(T))}.\notag
\end{align}
The constants depend only on shape regularity and the fixed boundary
partition. These estimates follow by combining polynomial scaling and
local averaging with patchwise Poincar\'e--Friedrichs inequalities
\cite[Chapter~3]{MalqvistPeterseim2020}; see also \cite{ErnGuermond2017}.

\begin{remark}[Choice of quasi-interpolation]\label{rem:coqi}
The choice of $I_H$ is not unique. Scott--Zhang-type operators and
coefficient-weighted or geometry-adapted operators designed for high-contrast
problems provide alternative choices; see, e.g.,
\cite{ScottZhang1990,EngwerHenningMalqvistPeterseim2019,
PeterseimScheichl2016,HellmanMalqvist2017}.
%Different choices can substantially affect approximation and localization errors and, through their approximation properties, the resulting convergence behavior.
\end{remark}

The kernel of this projection provides the fine-scale correction space.
We set $W=\ker I_H$, $W_h=W\cap V_h$, and $\Pi^{\rm f}=I-I_H$,
where $I$ is the identity. Since $I_H$ is a projection, we have the direct
decompositions $V=V_H\oplus W$ and $V_h=V_H\oplus W_h$.
To state the resolution condition, we introduce
\[
 \chi_{\rm res}=C_{\rm app}C_{\rm patch}(1)^{1/2}
                 \norm{kH}_{L^\infty(\Omega)}.
\]
\begin{lemma}[Kernel coercivity]\label{lem:kernel-coercivity}
If $\chi_{\rm res}\leq\chi_0<1$, then
\begin{equation*}
 \operatorname{Re}a(w,w)\geq c_W\norm w_k^2,\qquad w\in W,
\end{equation*}
where $c_W=(1-\chi_{\rm res}^2)/(1+\chi_{\rm res}^2)>0$.
\end{lemma}
\begin{proof}
Since $I_Hw=0$, \eqref{eq:IH-approx} and \eqref{eq:patch-overlap} give
$k\norm w_{L^2(\Omega)}\leq\chi_{\rm res}\norm{\nabla w}_{L^2(\Omega)}$.
The Robin term has zero real part, so
$\operatorname{Re}a(w,w)=\norm{\nabla w}_{L^2(\Omega)}^2-
 k^2\norm w_{L^2(\Omega)}^2$, which proves the claim.
\end{proof}
We assume this resolution condition throughout the analysis below.
In addition to kernel coercivity, the local interpolation estimates give
the energy stability bounds
\begin{equation}\label{eq:projection-stability}
 \norm{I_Hv}_k\leq C_I\norm v_k,\qquad
 \norm{\Pi^{\rm f}v}_k\leq(1+C_I)\norm v_k,
\end{equation}
where $C_I$ depends only on $\chi_0$ and fixed-order patch geometry.

\subsection{Correctors and the LOD problem}\label{sec:correctors}
With coercivity on the kernel available, we can define the LOD correctors.
The ideal primal and adjoint correctors $\Qir,\Qirs:V_h\to W_h$ are
determined by
\[
a(\Qir v_h,w)=a(v_h,w),\qquad
 a(w,\Qirs v_h)=a(w,v_h),\qquad w\in W_h.
\]
To localize these problems, we first split the form into element
contributions. For a union $D$ of coarse elements, let $a_D$ denote
\eqref{eq:form} restricted to $D$, with the Robin integral taken over
$\Gamma_R\cap\overline D$. For $T\in\mathcal T_H$ and
$\ell\in\mathbb N_0\cup\{\infty\}$, we set $N^\infty(T)=\Omega$ and
use the local kernel space
\[
 W_h(N^\ell(T))=\{w\in W_h:\supp w\subseteq\overline{N^\ell(T)}\}.
\]
The element correctors
$q_{T,\ell},q_{T,\ell}^*:V_H\to W_h(N^\ell(T))$ are defined by
\[
 a(q_{T,\ell}v_H,w)=a_T(v_H,w),\qquad
 a(w,q_{T,\ell}^*v_H)=a_T(w,v_H),
 \quad w\in W_h(N^\ell(T)).
\]
Kernel coercivity in Lemma \ref{lem:kernel-coercivity} ensures that both local problems are well posed.
Summing their solutions defines the operators
$\mathcal Q_{\ell}=\sum_Tq_{T,\ell}$ and
$\mathcal Q_{\ell}^*=\sum_Tq_{T,\ell}^*$ on $V_H$. At
$\ell=\infty$, these sums agree with the ideal correctors.
We now define
the transfer operators and the associated trial and test spaces by
\begin{align*}
 T_{\ell}&=I-\mathcal Q_{\ell}, \qquad
 X_{H,\ell}=T_{\ell}V_H,\\
 T_{\ell}^*&=I-\mathcal Q_{\ell}^*,\qquad 
 Y_{H,\ell}=T_{\ell}^*V_H.
\end{align*}
We seek the base LOD
approximation in the form $U_{H,\ell}=T_{\ell}u_{H,\ell}$, where
$u_{H,\ell}\in V_H$ solves the Petrov--Galerkin problem
\begin{equation}\label{eq:PG}
 a(T_{\ell}u_{H,\ell},T_{\ell}^*v_H)
 =F(T_{\ell}^*v_H),\qquad v_H\in V_H.
\end{equation}
For a load $F_\mu$, we denote this solution by $U_{H,\ell,\mu}$.
When no distinction among base, enriched, or compressed approximations is
needed, $U_\mu$ denotes the current accepted LOD approximation for
$F_\mu$.
For the ideal correctors, the defining equations give the orthogonality
relations $a(T_{\infty}v_h,w)=a(w,T_{\infty}^*v_h)=0$ for
$w\in W_h$. Kernel coercivity and continuity also give
\begin{equation}\label{eq:ideal-transfer-bound}
 \norm{T_{\infty}v_h}_k,\ \norm{T_{\infty}^*v_h}_k
 \leq C_F\norm{v_h}_k,\qquad C_F=1+C_a/c_W.
\end{equation}

\begin{lemma}[Ideal LOD stability]\label{prop:ideal-stability}
The ideal pair has inf--sup constant
\begin{equation}\label{eq:ideal-stability}
 \alpha_{H,\infty}:=
 \inf_{0\neq v\in X_{H,\infty}}\sup_{0\neq y\in Y_{H,\infty}}
 \frac{|a(v,y)|}{\norm v_k\norm y_k}
 \geq\gamma_h(k)/C_F.
\end{equation}
\end{lemma}
\begin{proof}
Let $v\in X_{H,\infty}$ and $y_h\in V_h$. By ideal orthogonality,
$a(v,y_h)=a(v,T_{\infty}^*y_h)$. Moreover, the corrected test belongs
to $Y_{H,\infty}$ because $T_{\infty}^*$ vanishes on $W_h$.
We can therefore take the supremum over $y_h$ and apply
\eqref{eq:ideal-transfer-bound} together with \eqref{eq:gamma-fine}
to obtain the stated lower bound.
\end{proof}

\section{A posteriori error control}\label{sec:certificates}
We now develop the error estimates that guide adaptation. We first
localize residuals on the fine-scale kernel, then apply the same
construction to the corrector-localization defect. LOD orthogonality
allows us to combine the two estimates into a bound for the error
relative to the fine-grid solution.

\subsection{Local residuals on the kernel}\label{sec:fine-grid-residual}
For each $z\in\mathcal N_H$, we introduce the vertex-patch space
$W_{h,z}=\{w\in W_h:\supp w\subseteq\overline{N^2(z)}\}$.
Throughout this subsection, all vertex-indexed sums run over
$z\in\mathcal N_H$, including the Dirichlet vertices.
The following decomposition allows us to work with these local spaces
without introducing constants that depend on the number of fine
elements in a coarse patch.

\begin{lemma}[Stable local decomposition]\label{thm:stable-decomposition}
Every $w\in W_h$ has a decomposition $w=\sum_z w_z$, $w_z\in W_{h,z}$,
such that
\begin{equation}\label{eq:stable-decomposition}
 \sum_z\norm{w_z}_k^2\leq C_{\rm sd}^2\norm w_k^2.
\end{equation}
Conversely, arbitrary $w_z\in W_{h,z}$ satisfy
\begin{equation}\label{eq:stable-synthesis}
 \norm{\sum_z w_z}_k^2\leq C_{\rm syn}^2\sum_z\norm{w_z}_k^2.
\end{equation}
The constants depend only on $\chi_0$, coarse- and fine-mesh shape regularity,
and fixed-order patch overlap.
\end{lemma}
\begin{proof}
Let $\lambda_z$ denote the coarse nodal hat functions, including those
at Dirichlet vertices, and let $\mathcal I_h$ be fine-grid nodal
interpolation on continuous functions. We choose
$w_z=\Pi^{\rm f}\mathcal I_h(\lambda_z w)$.
The locality of $I_H$ ensures that $w_z\in W_{h,z}$, while the partition
of unity gives $\sum_z w_z=\Pi^{\rm f}w=w$. To bound the energy of these
components, we note that $\lambda_z w$ is quadratic on each fine
triangle. Polynomial scaling and the product rule therefore yield
\[
 \sum_z\norm{\mathcal I_h(\lambda_z w)}_k^2
 \lesssim\norm{\nabla w}_{L^2(\Omega)}^2
          +\norm{H^{-1}w}_{L^2(\Omega)}^2+k^2\norm w_{L^2(\Omega)}^2
 \lesssim\norm w_k^2.
\]
In the last step, we used $I_Hw=0$ and \eqref{eq:IH-approx}.
Applying projection stability now proves \eqref{eq:stable-decomposition}.
For the converse estimate, we apply pointwise Cauchy--Schwarz and use
the bounded overlap of the patches to obtain
\eqref{eq:stable-synthesis}.
\end{proof}

We next associate a local Riesz representative with each patch. For
$G\in V_h^\star$, define $\xi_z(G)\in W_{h,z}$ and the combined
indicator $\eta(G)$ by
\begin{equation}\label{eq:as-local-riesz}
      \begin{aligned}
             b_k(\xi_z(G),w)&=G(w),\qquad w\in W_{h,z},\\
 \eta(G)^2&=\sum_z\norm{\xi_z(G)}_k^2.
      \end{aligned}
\end{equation}
When $G$ is the residual of an approximation $U\in V_h$ for a load
$F_\mu$, we use the notation
\begin{align*}
 R_{\mu,U}&=F_\mu-a(U,\cdot),\\
 \eta_{H,z,\mu}(U)&=\norm{\xi_z(R_{\mu,U})}_k,
 \qquad\eta_{H,\mu}(U)=\eta(R_{\mu,U}).
\end{align*}
We omit $\mu$ when discussing a single load and suppress $(U)$ when the
current approximation is clear. The stable decomposition shows that
these local indicators provide an equivalent norm of the residual
restricted to $W_h$.

\begin{theorem}[Residual localization and efficiency]
\label{thm:residual-localization}
For every $G\in V_h^\star$,
\begin{equation}\label{eq:residual-equivalence}
 C_{\rm syn}^{-1}\eta(G)\leq\norm G_{W_h^\star}
 \leq C_{\rm sd}\eta(G).
\end{equation}
Consequently, every $U\in V_h$ satisfies
\begin{equation}\label{eq:global-efficiency}
 \eta_H(U)\leq C_aC_{\rm syn}\norm{u_h-U}_k.
\end{equation}
Moreover,
$\eta_{H,z}(U)\leq C_{a,z}\norm{u_h-U}_{k,N^2(z)}$, where $C_{a,z}$ is
the local continuity constant.
\end{theorem}
\begin{proof}
We first use the decomposition in \eqref{eq:stable-decomposition} to
obtain 
\[|G(w)|\leq\eta(G)(\sum_z\norm{w_z}_k^2)^{1/2}
\leq C_{\rm sd}\eta(G)\norm w_k.\] 
Taking the supremum over $w\in W_h$
gives the upper bound in \eqref{eq:residual-equivalence}.
For the reverse inequality, \eqref{eq:as-local-riesz} and
\eqref{eq:stable-synthesis} give
\[
 \eta(G)^2=G\Bigl(\sum_z\xi_z(G)\Bigr)
 \leq C_{\rm syn}\norm G_{W_h^\star}\eta(G).
\]
This proves the reverse bound. Finally, we use
$R_U=a(u_h-U,\cdot)$ and continuity of $a$ to obtain global efficiency.
Applying the same argument to the restricted form on $N^2(z)$ gives the
patchwise estimate. Here $C_{a,z}\leq1$ for patches disjoint from
$\Gamma_R$, while boundary patches additionally involve a local trace
factor.
\end{proof}
\begin{remark}[Relation to preconditioning-based indicators]
Let $\mathscr M_hG=\sum_z\xi_z(G)$. Then
$\eta(G)^2=G(\mathscr M_hG)$, so $\eta(G)$ is the residual norm induced by
the additive Schwarz approximation $\mathscr M_h$ of the inverse Riesz map.
This is the preconditioning-based principle for a posteriori estimation
developed in \cite{LiZikatanov2021,LiZikatanov2025}. The nodal kernel spaces
and \cref{thm:stable-decomposition} provide its present subspace-correction
realization. It is also related to the smoother-type indicators of
\cite{LiShui2026}, which replace local inverse actions by inexpensive
smoothing steps on an enriched discrete space. Here the global Helmholtz
form is indefinite: the Schwarz construction controls the residual first on
the coercive fine-scale kernel $W_h$, while LOD Galerkin orthogonality,
localization control, and fine-grid inf--sup stability subsequently convert
that restricted estimate into the global bound of
\cref{thm:reliability}.
\end{remark}

\subsection{Corrector-localization error}\label{sec:corrector-certificate}
The same residual construction also measures the error caused by
truncating the corrector problems. On $V_H$, we define the primal and
adjoint defects by $\mathcal D_\ell=\Qir-\Qlr$ and
$\mathcal D_\ell^*=\Qirs-\Qlrs$. Complex conjugation gives
\begin{equation*}
 \mathcal Q_{\ell}v=\overline{\mathcal Q_{\ell}^*\overline v}
 \quad(v\in V_H),
 \qquad
 \delta_\ell:=\norm{\mathcal D_\ell}_{\mathcal L(V_H,V_h)}
              =\norm{\mathcal D_\ell^*}_{\mathcal L(V_H,V_h)},
\end{equation*}
where the operator norm is associated with the energy norm in \eqref{eq:energy}.
Although $\delta_\ell$
involves the ideal correctors, we can estimate it from localized
quantities alone. To do so, we introduce
\begin{equation*}
 G_\ell(v_H)(w)=\overline{a(w,T_{\ell}^*v_H)},\qquad
 \Theta_\ell=\sup_{0\neq v_H\in V_H}
              \frac{\eta(G_\ell(v_H))}{\norm{v_H}_k}.
\end{equation*}
The quantity $\Theta_\ell$ is a finite-dimensional operator norm
determined by the localized tests and the local Riesz problems.
Equivalently, its square is the largest generalized eigenvalue of the
associated residual and energy forms on $V_H$. More precisely, for a
basis $\{\phi_i\}$ of $V_H$, let
 \[
 (B_H)_{ij}=b_k(\phi_j,\phi_i),\qquad
 (R_\ell)_{ij}=\sum_z
 b_k\bigl(\xi_z(G_\ell(\phi_j)),\xi_z(G_\ell(\phi_i))\bigr).
 \]
Then $B_H$ is Hermitian positive definite, $R_\ell$ is Hermitian positive
semidefinite, and
$\Theta_\ell^2=\lambda_{\max}(R_\ell,B_H)$. The contributions to
$R_\ell$ are obtained from independent patch problems and can be assembled
in parallel, while the eigenvalue computation is global only on the coarse
space. This load-independent computation belongs to the offline adaptive
stage, is repeated only when the operator, mesh pair, or $\ell$ changes, and
is shared by the entire load family; it therefore adds no cost to subsequent
online solves and requires no global fine-grid solve. For large coarse spaces,
the largest eigenvalue can be computed iteratively. A factorization or
preconditioner for $B_H$ can be reused throughout this computation, as can
the local Riesz factorizations for multiple right-hand sides; the previous
eigenvector can also warm-start the solve after an adaptive update. The
following result relates this computable quantity to the corrector defect.

\begin{theorem}[Corrector estimate]\label{thm:corrector-certificate}
The localization defect satisfies
\begin{equation}\label{eq:corrector-certificate}
 \frac{\Theta_\ell}{C_aC_{\rm syn}}\leq\delta_\ell
 \leq\frac{C_{\rm sd}}{c_W}\Theta_\ell.
\end{equation}
\end{theorem}
\begin{proof}
For $w\in W_h$, ideal orthogonality gives
$G_\ell(v_H)(w)=\overline{a(w,\mathcal D_\ell^*v_H)}$.
Using kernel coercivity and continuity, we obtain
\[
 c_W\norm{\mathcal D_\ell^*v_H}_k
 \leq\norm{G_\ell(v_H)}_{W_h^\star}
 \leq C_a\norm{\mathcal D_\ell^*v_H}_k.
\]
We then apply \eqref{eq:residual-equivalence} and take the supremum
over $v_H$ to obtain both bounds.
\end{proof}
The computable estimate is complemented by the usual a priori decay
bound. For finite $\ell$, the standard LOD cut-off argument gives
\begin{equation}\label{eq:localization}
 \delta_\ell\leq C_{\rm loc}C_{\rm patch}(\ell+1)^{1/2}\beta^\ell,
 \qquad 0<\beta<1,
\end{equation}
with $C_{\rm loc}$ and $\beta$ independent of fine-mesh refinement depth
\cite{MalqvistPeterseim2014,Peterseim2017}.
For the discrete kernel, we use the kernel-preserving cut-off
$\Pi^{\rm f}\mathcal I_h(\vartheta w)$, where $\vartheta$ is a coarse
piecewise affine cut-off. The product estimate in
\cref{thm:stable-decomposition} controls this function, and kernel
coercivity yields contraction of the corrector tails across fixed-width
annuli. Summing the element contributions introduces the patch factor
in \eqref{eq:localization}. In contrast, the computable bound
\eqref{eq:corrector-certificate} does not require a growth law for this
factor.

\subsection{Two-sided fine-grid error estimate}\label{sec:global-reliability}
We are now ready to combine residual localization with LOD orthogonality.
The resulting estimate requires orthogonality only on the base test
space, a property that will allow us to use it for enriched solutions
as well.

\begin{theorem}[Reliability]\label{thm:reliability}
Let $U\in V_h$ satisfy $R_U(y)=0$ for every $y\in Y_{H,\ell}$.
If $\gamma_h(k)>0$ and kernel coercivity holds, then we have 
\begin{equation}\label{eq:global-reliability}
 \norm{u_h-U}_k\leq C_{\rm sd}
 \left(c_W^{-1}+\frac{C_IC_F}{\gamma_h(k)}\delta_\ell\right)\eta_H(U).
\end{equation}
In particular,
\begin{equation}\label{eq:fine-grid-interval}
 \frac{\eta_H(U)}{C_aC_{\rm syn}}
 \leq\norm{u_h-U}_k
 \leq C_{\rm sd}\left(c_W^{-1}
       +\frac{C_IC_FC_{\rm sd}}{\gamma_h(k)c_W}\Theta_\ell\right)\eta_H(U).
\end{equation}
\end{theorem}
\begin{proof}
We split the error as $e=u_h-U=e_W+e_H$, where
$e_W=\Qir e\in W_h$ and $e_H=T_{\infty}e\in X_{H,\infty}$.
The corrector equation and kernel coercivity immediately give
$\norm{e_W}_k\leq c_W^{-1}\norm{R_U}_{W_h^\star}$.
To estimate the coarse component, let $y_\infty=T_{\infty}^*v_H$.
Ideal orthogonality and the assumed orthogonality on $Y_{H,\ell}$ yield
\[
 a(e_H,y_\infty)=R_U(y_\infty)
 =R_U(y_\infty-T_{\ell}^*v_H)
 =-R_U(\mathcal D_\ell^*v_H).
\]
Since $v_H=I_Hy_\infty$, \eqref{eq:projection-stability} and
\eqref{eq:ideal-stability} imply
\[
 \norm{e_H}_k\leq\frac{C_IC_F}{\gamma_h(k)}\delta_\ell
                      \norm{R_U}_{W_h^\star}.
\]
We obtain \eqref{eq:global-reliability} by adding the bounds for the
two components and applying \eqref{eq:residual-equivalence}. Finally,
we combine this estimate with \eqref{eq:corrector-certificate} and
\eqref{eq:global-efficiency} to obtain the two-sided bound
\eqref{eq:fine-grid-interval}.
\end{proof}
\begin{remark}[Wavenumber robustness]
The only explicit wavenumber dependence in \eqref{eq:fine-grid-interval} is through
$\delta_\ell/\gamma_h(k)$, or through the computable upper bound
$\Theta_\ell/\gamma_h(k)$. Consequently, the reliability constant is
uniform in $k$ if
\begin{equation*}
 \delta_\ell\lesssim\gamma_h(k),
 \qquad\text{or, as a sufficient condition,}\quad
 \Theta_\ell\lesssim\gamma_h(k).
\end{equation*}
If $\gamma_h(k)^{-1}$ grows at most polynomially in $k$ and the patch factor
in \eqref{eq:localization} grows at most polynomially in $\ell$, the
exponential localization decay shows that $\ell=O(\log k)$ is sufficient.
\end{remark}

A useful feature of this result is that the approximation need not
belong to the base LOD trial space. It therefore also covers enriched
and compressed solutions whose test spaces contain $Y_{H,\ell}$,
provided we recompute the residual for the accepted solution.
When base orthogonality holds only approximately, we measure its defect by
\[
 \epsilon_0=\sup_{0\neq v_H\in V_H}
 \frac{|R_U(T_{\ell}^*v_H)|}{\norm{v_H}_k}.
\]
Repeating the proof then adds the term
$(C_IC_F/\gamma_h(k))\epsilon_0$ to the right-hand side of
\eqref{eq:global-reliability}.

\begin{remark}[Extension to three dimensions]
The analysis is presented in two dimensions, but the same construction
extends to conforming, shape-regular tetrahedral meshes in three
dimensions. In that case, triangles and edges are replaced by tetrahedra
and faces, respectively, and the local interpolation, trace, inverse, and
stable-decomposition estimates have their standard three-dimensional
counterparts. The resulting constants may additionally depend on the
spatial dimension and the three-dimensional shape-regularity bound.
\end{remark}

\subsection{Localized stability and oversampling}\label{sec:localized-perturbation}
Oversampling must preserve stability and control localization error. Let
$A_\ell(p_H,q_H)=a(T_{\ell}p_H,T_{\ell}^*q_H)$, and denote the
energy-norm inf--sup constant of $A_\infty$ on $V_H$ by
$\widehat\alpha_{H,\infty}$. Since
$I_HT_{\infty}=I_HT_{\infty}^*=I$ on $V_H$, ideal stability gives
\[
 \widehat\alpha_{H,\infty}\geq
 \frac{\alpha_{H,\infty}}{C_I^2}\geq\frac{\gamma_h(k)}{C_F C_I^2}.
\]
Ideal orthogonality cancels the first-order terms, giving the quadratic
perturbation identity
\begin{equation}\label{eq:coarse-perturbation}
 A_\ell(p_H,q_H)-A_\infty(p_H,q_H)
 =a(\mathcal D_\ell p_H,\mathcal D_\ell^*q_H).
\end{equation}

\begin{proposition}[Localized solvability and perturbation]
\label{prop:localized-perturbation}
If $C_a\delta_\ell^2<\widehat\alpha_{H,\infty}$, then
\eqref{eq:PG} is uniquely solvable and
\begin{equation}\label{eq:oversampling-perturbation}
 \norm{U_{H,\ell}-U_{H,\infty}}_k
 \leq C_{\rm os}(\delta_\ell)\delta_\ell\norm{u_h}_k,
\end{equation}
where $C_{\rm os}(\delta)=C_I+
 C_a(C_F+\delta)(1+C_I\delta)/(\widehat\alpha_{H,\infty}-C_a\delta^2)$.
\end{proposition}
\begin{proof}
By \eqref{eq:coarse-perturbation}, the inf--sup constant of $A_\ell$
is at least $\widehat\alpha_{H,\infty}-C_a\delta_\ell^2>0$.
Using $u_{H,\infty}=I_Hu_h$ and subtracting the coefficient equations gives
\[
 A_\ell(u_{H,\ell}-u_{H,\infty},q_H)
 =F(\mathcal D_\ell^*q_H)
  -a(\mathcal D_\ell u_{H,\infty},\mathcal D_\ell^*q_H).
\]
The inf--sup bound, $\norm F_{W_h^\star}\leq C_a\norm{u_h}_k$,
and $\norm{u_{H,\infty}}_k\leq C_I\norm{u_h}_k$ control the coefficient
difference. Since $U_{H,\ell}-U_{H,\infty}
=T_{\ell}(u_{H,\ell}-u_{H,\infty})+\mathcal D_\ell u_{H,\infty}$
and $\norm{T_{\ell}}\leq C_F+\delta_\ell$, these bounds yield the
stated estimate.
\end{proof}

Using \eqref{eq:corrector-certificate}, we impose
\begin{equation}
    \label{eq:stability-margin}
\frac{C_a^2}{c_W}\left(\frac{C_{\rm sd}}{c_W}\Theta_\ell\right)^2
 \leq\rho_{\rm st}\frac{\gamma_h(k)}{C_FC_I^2},
 \qquad
 \frac{C_{\rm sd}}{c_W}\Theta_\ell\leq C_{\rm rb}\gamma_h(k), 
\end{equation}
with $0<\rho_{\rm st}<1$ and $C_{\rm rb}>0$. The first condition also
ensures stability of the enriched Petrov--Galerkin problem in
Section~\ref{sec:enrichment-stability}; the second keeps the reliability
constant in \eqref{eq:global-reliability} uniform in $k$.
Both can be enforced by a single $k$-dependent threshold
$\tau_\Theta(k)$ for $\Theta_\ell$.

Stability alone does not ensure that localization error remains small
relative to the approximation error during refinement. To balance them,
define
\begin{equation}\label{eq:normalization}
 d_\mu=\max\{\norm{U_\mu}_k,d_{\min}\},\qquad d_{\min}>0,
\end{equation}
holding this scale fixed during each marking or compression decision.
For the base solution, evaluate $d_\mu$ and $\eta_{H,\mu}$ at
$U_{H,\ell,\mu}$. Combining \eqref{eq:oversampling-perturbation},
\eqref{eq:corrector-certificate}, and residual efficiency gives, for
bounded $C_{\rm os}(\delta_\ell)$ and positive denominators,
\begin{equation*}
 \frac{\norm{U_{H,\ell,\mu}-U_{H,\infty,\mu}}_k}
      {\norm{u_{h,\mu}-U_{H,\ell,\mu}}_k}
 \lesssim \Theta_\ell+\frac{d_\mu\Theta_\ell}{\eta_{H,\mu}}.
\end{equation*}
This motivates the balance indicator for a reference load $\mu_{\rm ref}$,
\begin{equation*}
 \widetilde q_\ell=
 \frac{d_{\mu_{\rm ref}}\Theta_\ell}{\eta_{H,\mu_{\rm ref}}}.
\end{equation*}
We increase $\ell$ whenever \eqref{eq:stability-margin} fails or
$\widetilde q_\ell$ exceeds its prescribed threshold.

\section{Family-aware adaptivity}\label{sec:family-enrichment}
This section combines the residual error control and localization criteria
of Section~\ref{sec:certificates} into an adaptive construction of shared
LOD approximation spaces for a family $\{F_\mu\}_{\mu\in\mathcal{P}}$ of parametrized loads. The method separately adjusts
the coarse mesh, the fine mesh used for local problems, the corrector
oversampling level, and a regional enrichment space. On a fixed mesh pair,
we solve a training family $\mathcal P_{\rm tr}\subset\mathcal P$ and repeatedly add residual-generated
directions from the member furthest from its regional target. Once the
targets are met or an enrichment limit is reached, we compress the shared
enrichment by energy-weighted POD with per-load acceptance checks.
Oversampling is increased when the stability or localization-balance
criteria require it. The mesh-update cycle then combines one coarse
refinement with prescribed fine-only refinements, repeating the family
solve after each mesh update. We construct compatible test spaces for the
enrichment and establish stability of the enriched approximation.

The enrichment region $D\subseteq\Omega$ is prescribed in advance to
contain localized features that are expected to benefit from additional
basis functions, such as a neighborhood of a known reentrant corner.
It is an input to the algorithm and remains fixed during the adaptive run.
Only coarse-vertex patches $N^2(z)$ contained in $\overline D$ contribute
to the unprojected regional corrections. Within this region, the training
residuals determine which load supplies the next enrichment directions.
Coarse and fine marking continue to use residual information throughout
$\Omega$, so the choice of $D$ restricts regional enrichment
without restricting mesh refinement to that region.

\subsection{Regional residual enrichment}\label{sec:residual-enrichment}
We first fix a mesh pair and an oversampling level satisfying
\eqref{eq:stability-margin}. Let
$W_{h,D}=\{w\in W_h:\operatorname{supp}w\subseteq\overline D\}$.
For the prescribed region $D$, define
\begin{align*}
 \mathcal Z_D&=\{z\in\mathcal N_H:N^2(z)\subseteq\overline D\},&
 \mathscr M_{h,D}G&=\sum_{z\in\mathcal Z_D}\xi_z(G),\\
 \eta_{D,\mu}(U)^2&=\sum_{z\in\mathcal Z_D}
                 \norm{\xi_z(R_{\mu,U})}_k^2.
\end{align*}
The unprojected regional correction $\psi=\mathscr M_{h,D}R_{\mu,U}$ belongs to $W_{h,D}$,
and satisfies
$R_{\mu,U}(\psi)=\eta_{D,\mu}(U)^2$.
Thus $\psi$ is an additive Schwarz subspace correction
\cite{Xu1992,KornhuberPeterseimYserentant2018} obtained from the local
problems already used by the indicator. We use the regional indicator
to guide enrichment within $D$, while retaining the
global indicator $\eta_{H,\mu}$ for error control and coarse marking.

To obtain several directions from one residual, we apply the regional
Schwarz correction repeatedly after removing components already present
in the trial space. For a subspace $Z\subseteq V_h$, define
\[
 Z^{\perp_k}:=\{v\in V_h:b_k(v,z)=0\ \text{for all }z\in Z\},
\]
and let $P_Z^k:V_h\to Z$ denote the $b_k$-orthogonal projection. Let
$\mathscr A_h:V_h\to V_h^\star$ satisfy
$(\mathscr A_hv)(w)=a(v,w)$.
For a current enrichment $\mathcal E\subseteq X_{H,\ell}^{\perp_k}$, set
$S=X_{H,\ell}\oplus\mathcal E$.
We generate a block of at most $s\in\mathbb N$ new directions
through the space
\begin{equation}\label{eq:enrichment-block}
 \operatorname{span}\{g,K_Sg,\ldots,K_S^{s-1}g\},
\end{equation}
where $g=(I-P_S^k)\mathscr M_{h,D}R_{\mu,U}$ and $K_S=(I-P_S^k)\mathscr M_{h,D}\mathscr A_h$.
The block in \eqref{eq:enrichment-block} is a regional, projected variant
of the reduced Krylov basis construction in \cite{LiZikatanovZuo2025}.

The projection removes components already represented in $S$ without
prescribing an orthogonalization procedure. Although this projection may
change support and kernel membership, it leaves the enriched trial space
unchanged because
$X_{H,\ell}+\operatorname{span}\{(I-P_{X_{H,\ell}}^k)\psi\}
 =X_{H,\ell}+\operatorname{span}\{\psi\}$.
The regional support restriction therefore applies to the unprojected
regional corrections, rather than necessarily to the final trial modes.

\begin{remark}[Relation to localized a posteriori error control]
\label{rem:local-error-adaptivity}
Localized a posteriori estimates have been used to drive adaptive finite
element methods for controlling local energy or pointwise errors, as in \cite{Demlow2010,GuoLi2026}.
These approaches target accuracy in a prescribed subregion while
accounting for error contributions from the rest of the domain through
global pollution terms reflecting the influence of errors outside the
target region.
In the present method, $D$ instead specifies where residual corrections
are generated to construct a shared enrichment space.
The shared enrichment provides a small common set of correction directions
that improves approximation for multiple loads and can be reused for additional loads. Global residual
information continues to guide coarse and fine mesh refinement throughout
$\Omega$, and the a posteriori estimate controls the global energy error
relative to the fine-grid solution. 
\end{remark}

\subsection{Training and compression}
\label{sec:kernel-lifted-tests}
\label{sec:family-enrichment-compression}

We construct tests for the added trial directions by a kernel lifting
built from the existing localized LOD trial operator. Define
\begin{equation}\label{eq:kernel-lifted-test-map}
 \mathscr J_\ell=I-T_{\ell}I_H:V_h\to W_h.
\end{equation}
Since $I_HT_{\ell}=I$ on $V_H$, this map is a projection onto $W_h$
along $X_{H,\ell}$.

For an $r$-dimensional enrichment
$\mathcal E\subseteq X_{H,\ell}^{\perp_k}$, set
$\widehat{\mathcal E}=\mathscr J_\ell\mathcal E$ and use the trial and test spaces
\begin{equation}\label{eq:enriched-spaces}
 X_{H,\ell}\oplus\mathcal E
 =X_{H,\ell}\oplus\widehat{\mathcal E},
 \qquad
 Y_{H,\ell}\oplus\widehat{\mathcal E},
\end{equation}
respectively. Indeed, $v-\mathscr J_\ell v\in X_{H,\ell}$, and the restriction of
$\mathscr J_\ell$ to $\mathcal E$ is injective. Thus the original trial
space and its energy-orthogonal representation are retained. Consequently,
for every finite-dimensional $\mathcal E\subseteq X_{H,\ell}^{\perp_k}$,
the enriched trial and test spaces have equal dimensions:
\[
 \dim(X_{H,\ell}\oplus\mathcal E)
 =\dim(Y_{H,\ell}\oplus\widehat{\mathcal E})
 =\dim V_H+\dim\mathcal E.
\]
In this way,
the enriched problem becomes: find $U_\mu(\mathcal E)\in X_{H,\ell}\oplus\mathcal E$
such that
\begin{equation}\label{eq:enriched}
   a(U_\mu(\mathcal E),y)=F_\mu(y),
 \qquad y\in Y_{H,\ell}\oplus\widehat{\mathcal E}.
\end{equation}

An important advantage of the added tests is that they can recover the
local support for the enrichment constructed above. Inductively, the
projected enrichment procedure preserves
$X_{H,\ell}+\mathcal E\subseteq X_{H,\ell}+W_{h,D}$, since subtracting
a projection onto the current trial space does not leave
$X_{H,\ell}+W_{h,D}$. As $\mathscr J_\ell$ annihilates $X_{H,\ell}$ and
is the identity on $W_h$, this
implies $\widehat{\mathcal E}\subseteq W_{h,D}$. The same property holds
for every compressed subspace of the full enrichment.

For error control, we retain the inclusion
$Y_{H,\ell}\subseteq Y_{H,\ell}\oplus\widehat{\mathcal E}$.
After recomputing the residual for the accepted enriched solution,
\eqref{eq:fine-grid-interval} therefore applies without any new indicator.
All enriched solutions and residual checks in the training and compression
steps below use the tests in \eqref{eq:enriched-spaces}. If the enriched equations are solved
inexactly, the defect of base orthogonality is treated by the additional
$\epsilon_0$ term from Section~\ref{sec:global-reliability}.

For the fixed base family, we prescribe the per-load targets
\begin{equation}\label{eq:regional-training-target}
 t_\mu=\max\{t_{\min},\tau_D\eta_{D,\mu}(U_{H,\ell,\mu})\},\qquad
 t_{\min}>0,\quad 0<\tau_D<1.
\end{equation}
At each enrichment step, the member with the largest ratio
$\eta_{D,\mu}(U_\mu(\mathcal E))/t_\mu$ supplies the next block in
\eqref{eq:enrichment-block}. Our target is to satisfy the regional
criterion for every training member,
\begin{equation}\label{eq:regional-stopping-criterion}
 \max_{\mu\in\mathcal P_{\rm tr}}
 \frac{\eta_{D,\mu}(U_\mu(\mathcal E))}{t_\mu}\leq1.
\end{equation}

Let $\mathcal E_{\rm full}$ be the resulting space  
and $U_\mu^{\rm full}$ be the corresponding enriched solutions.
To identify the directions shared most strongly by the training family,
we use $d_\mu$ from \eqref{eq:normalization} at these solutions and
form the normalized snapshots
\begin{equation*}
 w_\mu=\frac{U_\mu^{\rm full}-U_{H,\ell,\mu}}{d_\mu}.
\end{equation*}
Then we perform energy-weighted POD to select an $r$-dimensional subspace
$\mathcal E_r\subseteq\mathcal E_{\rm full}$ that minimizes their
average projection error. Let $\sigma_j$ be the singular values, in
decreasing order, of the projected snapshot coordinates in an
energy-orthonormal basis of $\mathcal E_{\rm full}$. The POD
optimality relation is, with $P_{\mathcal E}^k$ denoting the
$b_k$-orthogonal projection onto $\mathcal E$,
\begin{equation}\label{eq:pod-optimality}
 \begin{aligned}
 \min_{\substack{\mathcal E\subseteq\mathcal E_{\rm full}\\\dim\mathcal E=r}}
 \frac1M\sum_{\mu\in\mathcal P_{\rm tr}}\norm{(I-P_{\mathcal E}^k)w_\mu}_k^2
 ={}&\frac1M\sum_{\mu\in\mathcal P_{\rm tr}}
       \norm{(I-P_{\mathcal E_{\rm full}}^k)w_\mu}_k^2
       +\frac1M\sum_{j>r}\sigma_j^2.
 \end{aligned}
\end{equation}
To see this, we split each snapshot into its projection onto
$\mathcal E_{\rm full}$ and an orthogonal component. The latter does
not depend on the chosen subspace, while the former leads to the best
rank-$r$ approximation problem \cite{KunischVolkwein2002}.

Although POD minimizes the average projection error, it does not
control each training member separately. We therefore supplement it
with two per-load checks. A tested rank is accepted only if the
compressed solution $U_{\mu,r}:=U_\mu(\mathcal E_r)$, recomputed with
test space $Y_{H,\ell}\oplus\mathscr J_\ell\mathcal E_r$, satisfies
\begin{align}
 \norm{U_\mu^{\rm full}-U_{\mu,r}}_k
 &\leq\varepsilon_{\rm dist}t_\mu,\label{eq:distance-criterion}\\
 \eta_{D,\mu}(U_{\mu,r})
 &\leq\max\{c_t t_\mu,c_\rho\eta_{D,\mu}(U_\mu^{\rm full})\},
 \label{eq:residual-criterion}
\end{align}
for every training member, where $\varepsilon_{\rm dist}>0$ and
$c_t,c_\rho\geq1$ are prescribed. The stability bound established in
\cref{prop:augmented-quasioptimality} in
Section~\ref{sec:enrichment-stability} applies to every
$\mathcal E_r\subseteq\mathcal E_{\rm full}$, so no additional
rank-dependent oversampling check is needed. We retain the smallest tested rank
that satisfies both conditions. 

\subsection{Stability of kernel-lifted enrichment}
\label{sec:enrichment-stability}

The following result establishes well-posedness and localization-explicit
quasi-optimality of the enriched problem \eqref{eq:enriched}, with bounds
independent of the choice and dimension of the enrichment space.
To state the result, introduce
\begin{equation*}
 \lambda_W=\frac{C_a}{c_W},\qquad
 L_\ell=C_F+\lambda_W\delta_\ell,\qquad
 m_\ell=\widehat\alpha_{H,\infty}
                -\frac{C_a^2}{c_W}\delta_\ell^2.
\end{equation*}

\begin{proposition}[Stability and quasi-optimality of kernel-lifted enrichment]
\label{prop:augmented-quasioptimality}
Assume $\gamma_h(k)>0$, kernel coercivity, and
$m_\ell>0$, the latter being ensured by \eqref{eq:stability-margin}. Then, for every
finite-dimensional $\mathcal E\subseteq X_{H,\ell}^{\perp_k}$,
the enriched problem is uniquely solvable, and its inf--sup
constant satisfies 
\begin{equation}\label{eq:kernel-enrichment-infsup}
 \alpha_{\mathcal E}:=
 \inf_{0\neq v\in X_{H,\ell}\oplus\mathcal E}
 \sup_{0\neq y\in Y_{H,\ell}\oplus\widehat{\mathcal E}}
 \frac{|a(v,y)|}{\norm v_k\norm y_k}
 \geq
 \frac{\min\{m_\ell,c_W\}}{1+L_\ell^2}.
\end{equation}
Moreover, the localization-explicit quasi-optimality estimate holds
\begin{equation}\label{eq:enriched-quasioptimality}
 \norm{u_{h,\mu}-U_\mu(\mathcal E)}_k
 \leq C_{\rm qo,\ell}
      \inf_{v\in X_{H,\ell}\oplus\mathcal E}\norm{u_{h,\mu}-v}_k,
\end{equation}
where $C_{\rm qo,\ell}
 =\lambda_W(\lambda_W+C_I\delta_\ell)
       \left(1+C_aL_\ell\delta_\ell/m_\ell\right)$.
\end{proposition}

\begin{proof}
Set $\mathcal F:=\widehat{\mathcal E}
=\mathscr J_\ell\mathcal E\subset W_h$. Let
$P,P^\dagger:W_h\to\mathcal F$ be the primal and adjoint Ritz projections,
\[
 a(Pw,f)=a(w,f),\qquad
 a(f,P^\dagger w)=a(f,w),
 \qquad f\in\mathcal F.
\]
Kernel coercivity and continuity imply
\[
 \norm{I-P},\ \norm{I-P^\dagger}\leq\lambda_W.
\]
Define
\[
 \widetilde T_\ell
 =T_{\infty}+(I-P)\mathcal D_\ell,
 \qquad
 \widetilde T_\ell^*
 =T_{\infty}^*+(I-P^\dagger)\mathcal D_\ell^*.
\]
Since $T_{\ell}-\widetilde T_\ell$ and
$T_{\ell}^*-\widetilde T_\ell^*$ take values in $\mathcal F$,
\[
 \begin{aligned}
 X_{H,\ell}\oplus\mathcal E
 &=\widetilde T_\ell V_H\oplus\mathcal F,\\
 Y_{H,\ell}\oplus\widehat{\mathcal E}
 &=\widetilde T_\ell^*V_H\oplus\mathcal F.
 \end{aligned}
\]
The Ritz identities and ideal orthogonality yield
\[
 a(\widetilde T_\ell p_H,f)
 =a(f,\widetilde T_\ell^*q_H)=0,
 \qquad f\in\mathcal F,
\]
so the enriched form is block diagonal. Its coarse block
\[
 S_{\ell,\mathcal F}(p_H,q_H)
 =a(\widetilde T_\ell p_H,\widetilde T_\ell^*q_H)
\]
satisfies
\[
 S_{\ell,\mathcal F}(p_H,q_H)-A_\infty(p_H,q_H)
 =a((I-P)\mathcal D_\ell p_H,\mathcal D_\ell^*q_H),
\]
and hence
\[
 |S_{\ell,\mathcal F}(p_H,q_H)-A_\infty(p_H,q_H)|
 \leq \frac{C_a^2}{c_W}\delta_\ell^2
       \norm{p_H}_k\norm{q_H}_k.
\]
Thus the coarse and kernel blocks have inf--sup constants at least
$m_\ell$ and $c_W$, respectively.

Moreover, it follows from
$\|\widetilde T_\ell\|,
\|\widetilde T_\ell^*\|\leq L_\ell$ that 
\[
 \|\widetilde T_\ell p_H+f\|_k
 \leq \sqrt{1+L_\ell^2}
       \bigl(\norm{p_H}_k^2+\norm f_k^2\bigr)^{1/2},
\]
and likewise on the test side. This proves
\eqref{eq:kernel-enrichment-infsup} and unique solvability.

For quasi-optimality, let $e=u_{h,\mu}-U_\mu(\mathcal E)$ and define
\[
 \mathscr B_\ell=(I-P)(\Qir-\mathcal D_\ell I_H):V_h\to W_h.
\]
Since $\Qir T_\ell=\mathcal D_\ell$ and $I_HT_\ell=I$ on $V_H$,
while $\Qir f=f$ and $I_Hf=0$ for $f\in\mathcal F$, this operator
annihilates $X_{H,\ell}\oplus\mathcal F
=X_{H,\ell}\oplus\mathcal E$.
Kernel coercivity gives $\norm{\Qir}\leq\lambda_W$, and hence
\[
 \norm{\mathscr B_\ell}
 \leq\lambda_W(\lambda_W+C_I\delta_\ell).
\]
Writing $w=\mathscr B_\ell e$, we therefore obtain, for every
$v\in X_{H,\ell}\oplus\mathcal E$,
\[
 \norm w_k
 =\norm{\mathscr B_\ell(u_{h,\mu}-v)}_k
 \leq\lambda_W(\lambda_W+C_I\delta_\ell)
       \norm{u_{h,\mu}-v}_k.
\]

Galerkin orthogonality gives $a(e,f)=0$ for $f\in\mathcal F$.
By ideal orthogonality, $a(\Qir e,f)=a(e,f)=0$, so
$P\Qir e=0$. With $q_H=I_He$, this yields the decomposition
\[
 e=\widetilde T_\ell q_H+w,\qquad Pw=0.
\]
Testing against $\widetilde T_\ell^*v_H$ and using the block
orthogonality gives
\[
 S_{\ell,\mathcal F}(q_H,v_H)
 =-a(w,\widetilde T_\ell^*v_H)
 =-a(w,\mathcal D_\ell^*v_H).
\]
The last identity follows from ideal orthogonality and
$a(w,f)=0$ for $f\in\mathcal F$. The coarse inf--sup bound now implies
\[
 \norm{q_H}_k\leq\frac{C_a\delta_\ell}{m_\ell}\norm w_k.
\]
Consequently,
\[
 \norm e_k
 \leq\left(1+\frac{C_aL_\ell\delta_\ell}{m_\ell}\right)\norm w_k.
\]
Combining these estimates and taking the infimum over $v$ proves
\eqref{eq:enriched-quasioptimality}.
\end{proof}

\begin{remark}[Localization perturbation for enriched solutions]
\label{rem:enriched-localization-perturbation}
The localization balance can also be evaluated at the enriched solution.
Indeed, fix the current enrichment and set
$\mathcal F:=\widehat{\mathcal E}\subset W_h$. Let
$U_{\mu,\infty}^{\mathcal F}\in
T_{\infty}V_H\oplus\mathcal F$ denote the ideal enriched solution
with test space $T_{\infty}^*V_H\oplus\mathcal F$. Using the same Ritz
projections as above, the localized and ideal enriched problems have the
same $\mathcal F$-component, whereas their coarse blocks differ only
through $\mathcal D_\ell$. Since the ideal coarse coefficient is
$I_Hu_{h,\mu}$, subtraction of the two coarse equations gives
\[
 \norm{U_\mu(\mathcal E)-U_{\mu,\infty}^{\mathcal F}}_k
 \leq C_{\rm loc,\ell}^{\rm enr}\,
       \delta_\ell\norm{u_{h,\mu}}_k,
\]
where
\[
 C_{\rm loc,\ell}^{\rm enr}
 =\lambda_W\left(
   C_I+\frac{C_aL_\ell(1+C_I\delta_\ell)}{m_\ell}
   \right).
\]
Combining this estimate with residual efficiency and the computable
bound for $\delta_\ell$ yields, whenever
$C_{\rm loc,\ell}^{\rm enr}$ remains bounded,
\[
 \frac{
   \norm{U_\mu(\mathcal E)-U_{\mu,\infty}^{\mathcal F}}_k}
 {\norm{u_{h,\mu}-U_\mu(\mathcal E)}_k}
 \lesssim
 \Theta_\ell+
 \frac{d_\mu\Theta_\ell}
      {\eta_{H,\mu}(U_\mu(\mathcal E))}.
\]
Hence the localization-balance indicator may be evaluated using the
accepted enriched, and in particular compressed, solution. 
\end{remark}

\subsection{Coarse and fine refinement}\label{sec:family-marking}
Once the family solutions have been accepted, we use their global
residuals to mark the coarse mesh. We first distribute each squared
vertex indicator over the adjacent coarse elements and average the
normalized contributions over the family:
\begin{equation*}
 m_{T,\mu}=\sum_{z\in\mathcal N_H\cap\overline T}
             \frac{\eta_{H,z,\mu}^2}{m_z},\qquad
 m_z=\#\{T\in\mathcal T_H:z\in\overline T\},\qquad
 \overline m_T=\frac1M\sum_{\mu\in\mathcal P_{\rm tr}}
                         \frac{m_{T,\mu}}{d_\mu^2}.
\end{equation*}
By construction, $\sum_Tm_{T,\mu}=\eta_{H,\mu}^2$.
We choose a smallest D\"orfler set $\mathcal M$
\cite{Doerfler1996} satisfying
\begin{equation}\label{eq:coarse-dorfler}
 \sum_{T\in\mathcal M}\overline m_T
 \geq\theta_H\sum_T\overline m_T,\qquad 0<\theta_H\leq1.
\end{equation}
To ensure that the average does not overlook the least-resolved
load, we choose $\mu_H\in\operatorname*{arg\,max}_{\mu\in\mathcal P_{\rm tr}}
\eta_{H,\mu}/d_\mu$ and supplement $\mathcal M$ by as few
elements as necessary. The resulting set $\mathcal M_H$ must also
satisfy
\begin{equation}\label{eq:worst-family-supplement}
 \sum_{T\in\mathcal M_H}m_{T,\mu_H}
 \geq\theta_H\sum_Tm_{T,\mu_H}.
\end{equation}
The marked set therefore captures the required fraction of both the
family-averaged residual and the residual of its least-resolved load,
without imposing a separate bulk condition on every load.

Coarse refinement controls the error relative to the fine-grid
solution. We also refine the fine mesh to address its own discretization
error, as indicated by the splitting
\begin{equation*}
 \norm{u_\mu-U_\mu}_k\leq
 \norm{u_\mu-u_{h,\mu}}_k+\norm{u_{h,\mu}-U_\mu}_k.
\end{equation*}
For volume loads $f_\mu\in L^2(\Omega)$ and each element
$K\in\mathcal T_h$, we use the residual indicator
\begin{align*}
 \zeta_{K,\mu}^2={}&h_K^2\norm{f_\mu+k^2U_\mu}_{L^2(K)}^2
 +\frac12\sum_{e\subset\partial K\cap\Omega}
        h_e\norm{\llbracket\partial_\nu U_\mu\rrbracket}_{L^2(e)}^2\\
 &+\sum_{e\subset\partial K\cap\Gamma_N}
        h_e\norm{\partial_\nu U_\mu}_{L^2(e)}^2
 +\sum_{e\subset\partial K\cap\Gamma_R}
        h_e\norm{\partial_\nu U_\mu-\ii kU_\mu}_{L^2(e)}^2.
\end{align*}
Here $\llbracket\partial_\nu U_\mu\rrbracket$ denotes the jump of the normal derivative
across an interior edge, defined as the sum of the outward normal
derivatives from the two adjacent elements.
This is the standard residual indicator for the conforming fine-grid
discretization. Here we only use
$\zeta_{K,\mu}$ to guide fine refinement.
For fine marking, we apply the same family rule to the weights
$\zeta_{K,\mu}^2/d_\mu^2$, select the least-resolved load using
$(\sum_K\zeta_{K,\mu}^2)^{1/2}/d_\mu$, and use the bulk fraction
$0<\theta_h\leq1$.

Algorithm~\ref{alg:family-solve} gives the load-family solve on a fixed
mesh pair, and Algorithm~\ref{alg:coarse-fine-cycle} gives one coarse--fine
refinement cycle. The cycle first refines the coarse mesh, adjusting the
fine mesh as needed to preserve nestedness, and then performs
$m_{\rm ref}\in\mathbb N$ fine-only updates with the coarse mesh fixed.
After each mesh update, Algorithm~\ref{alg:family-solve} recomputes the
family solutions, enrichment, and residual indicators before the next
marking step. The fine-only updates use marking fractions
$\theta_h^{(j)}\in(0,1]$, $j=1,\ldots,m_{\rm ref}$.

\begin{algorithm}[H]
\caption{Load-family solve on a fixed mesh pair}
\label{alg:family-solve}
\small
\begin{algorithmic}[1]
\Require $\mathcal T_H\preceq\mathcal T_h$, $\ell$, $\mathcal P_{\rm tr}$,
$\mu_{\rm ref}\in\mathcal P_{\rm tr}$, $D$, $d_{\min}$, $\tau_q$,
$\tau_\Theta(k)$, $\ell_{\max}$, $t_{\min}$, $\tau_D$, $s$,
$\varepsilon_\mathrm{dist}$, $c_t$, $c_\rho$, a prescribed enrichment limit,
and optional inherited unprojected corrections.
\Ensure Family solutions, retained enrichment and unprojected corrections,
updated $\ell$, and residual indicators.
\State Construct $X_{H,\ell}$ and $Y_{H,\ell}$, and compute $\Theta_\ell$.
\While{$\Theta_\ell>\tau_\Theta(k)$ and $\ell<\ell_{\max}$}
  \State Increase $\ell$, rebuild the base spaces, and recompute $\Theta_\ell$.
\EndWhile
\State If $\Theta_\ell>\tau_\Theta(k)$, stop and report that the stability threshold is not met.
\State Solve \eqref{eq:PG} for every $\mu\in\mathcal P_{\rm tr}$.
\State Set $t_\mu$ by \eqref{eq:regional-training-target}. Transfer inherited
corrections to obtain $\mathcal E\subseteq X_{H,\ell}^{\perp_k}$, or set
$\mathcal E=\{0\}$.
\State Form $\mathscr J_\ell\mathcal E$ by \eqref{eq:kernel-lifted-test-map}
and solve \eqref{eq:enriched} for $U_\mu(\mathcal E)$.
\While{\eqref{eq:regional-stopping-criterion} fails and no enrichment limit is reached}
  \State Choose $\mu_D\in\operatorname*{arg\,max}_{\mu\in\mathcal P_{\rm tr}}
  \eta_{D,\mu}(U_\mu(\mathcal E))/t_\mu$.
  \State Enlarge $\mathcal E$ by the block \eqref{eq:enrichment-block}
  generated from $R_{\mu_D,U_{\mu_D}(\mathcal E)}$.
  \State Form the added tests by \eqref{eq:kernel-lifted-test-map}, solve for
  the enriched family solutions, and update the regional indicators.
\EndWhile
\State Set $\mathcal E_{\rm full}=\mathcal E$. Compress by
\eqref{eq:pod-optimality}, retaining the smallest tested rank satisfying
\eqref{eq:distance-criterion}--\eqref{eq:residual-criterion}, using
$Y_{H,\ell}\oplus\mathscr J_\ell\mathcal E_r$ for each tested rank.
\State Recompute $d_\mu$, $\eta_{H,\mu}(U_{\mu,r})$, and
$\widetilde q_\ell$ from the retained family, reusing $\Theta_\ell$.
\If{$\ell<\ell_{\max}$ and $\widetilde q_\ell>\tau_q$}
  \State Increase $\ell$ by one, retain the unprojected corrections, and repeat
  from Step~1 with updated base spaces, tests, and regional targets.
\EndIf
\State Retain the current family, enrichment, and indicators. 
\end{algorithmic}
\end{algorithm}

\begin{algorithm}[H]
\caption{One coarse--fine refinement cycle}
\label{alg:coarse-fine-cycle}
\small
\begin{algorithmic}[1]
\Require Current mesh pair, family solutions, $\ell$, retained unprojected
corrections, and marking fractions $\theta_H$, $\theta_h^{(1)},\ldots,
\theta_h^{(m_{\rm ref})}$.
\Ensure Updated nested mesh pair, family solutions, enrichment, and indicators.
\State Mark the coarse set $\mathcal M_H$ by
\eqref{eq:coarse-dorfler}--\eqref{eq:worst-family-supplement}.
\State Refine $\mathcal T_H$ on $\mathcal M_H$ by newest vertex bisection,
and refine $\mathcal T_h$ as needed to preserve conformity and
$\mathcal T_H\preceq\mathcal T_h$.
\State Apply Algorithm~\ref{alg:family-solve} on the updated mesh pair.
\For{$j=1,\ldots,m_{\rm ref}$}
  \State Keep the coarse mesh fixed and recompute the fine indicators
  from the current family solutions.
  \State Mark with fraction $\theta_h^{(j)}$, refine the fine mesh,
  and preserve conformity and nestedness.
  \State Apply Algorithm~\ref{alg:family-solve} on the updated pair.
\EndFor
\State Retain the resulting state for the next coarse-marking step.
\end{algorithmic}
\end{algorithm}

\section{Numerical experiments}\label{sec:numerics}

\subsection{Error measures and comparison setting}
We present two experiments using adaptive and uniform conforming $P_1$
finite elements (AFEM and UFEM) and LOD for comparison. The reentrant-corner
experiment examines family-aware construction, regional enrichment, and
oversampling; the second experiment studies wavenumber dependence. 
We report the relative energy error against the exact solution,
\[
 E_\mu=\frac{\norm{u_\mu-U_\mu}_k}{\norm{u_\mu}_k},
\]
where $u_\mu$ denotes the manufactured exact solution, used to define the
load and evaluate errors, and $U_\mu$ denotes the numerical solution
produced by the corresponding method.
We fit $E\approx C N_{\mathrm{on}}^{-p}$, where $p$ is the observed decay
order and $N_{\mathrm{on}}$ denotes the number of online degrees of freedom (dof).
For the initial ALOD mesh pair, we define the scales by the maximum
element diameters,
\[
 H_0=\max_{T\in\mathcal T_{H,0}}\operatorname{diam}T,\qquad
 h_0=\max_{K\in\mathcal T_{h,0}}\operatorname{diam}K.
\]
The reentrant-corner experiment starts with $H_0=\sqrt2/8$ and
$h_0=\sqrt2/32$.
The wavenumber-dependent initial scales are specified in
Section~\ref{sec:num-wavenumber}.
Table~\ref{tab:adaptive-config} collects the main ALOD settings.

\begin{table}[!htbp]
\caption{Principal ALOD settings. Initial mesh diameters are listed as
$(H_0,h_0)$ pairs.}
\label{tab:adaptive-config}\centering\small
\begin{tabular}{@{}lcc@{}}\toprule
Setting & Reentrant corner & Wavenumber study\\\midrule
Wavenumber $k$ & 16 & $8,16,32,64,128$\\
Initial mesh diameters $H_0/h_0$ & $(\sqrt2/8,\sqrt2/32)$ &
$(2\sqrt2/k,\sqrt{8/k^3})$\\
Coarse marking $\theta_H$ & \multicolumn{2}{c}{$0.15$}\\
Fine refinements per cycle & \multicolumn{2}{c}{$m_{\rm ref}=2$}\\
Fine marking fractions & \multicolumn{2}{c}{$(0.3,0.2)$}\\
Initial $\ell$ / maximum $\ell_{\max}$ & $2/4$ & $2/5$\\
Coarse-cycle limit & 28 & 40\\
Balance threshold $\tau_q$ & \multicolumn{2}{c}{$0.3$}\\
Defect threshold $\tau_\Theta(k)$ & $0.2$ & $3.2/k$\\
Scale floor $d_{\min}$ & \multicolumn{2}{c}{$10^{-12}$}\\
\bottomrule\end{tabular}
\end{table}

\subsection{A reentrant corner}\label{sec:num-lshape}
We consider the L-shaped domain $\Omega=(-1,1)^2\setminus([0,1]\times[-1,0])$.
In polar coordinates $(\varrho,\theta)$ about the reentrant corner,
$0<\theta<3\pi/2$. The corner edges carry Dirichlet conditions,
$\Gamma_N=\varnothing$, and the remaining boundary carries Robin conditions:
\[
 \Gamma_D=\{(x,0):0<x<1\}\cup\{(0,y):-1<y<0\},\qquad
 \Gamma_R=\partial\Omega\setminus\overline{\Gamma_D}.
\]
At $k=16$, the parameter $\mu=(\mu_c,\mu_\alpha,\mu_x,\mu_y,\mu_\varphi)$
varies the singular-component amplitude, packet amplitude, center, and phase:
\begin{align*}
 b_\partial(x,y)&=(1-x^2)^2(1-y^2)^2,\qquad
 u_{\rm sing}=b_\partial\varrho^{2/3}\sin(2\theta/3),\\*
 A_\mu(x,y)&=\frac{xyb_\partial(x,y)}{\mu_x\mu_yb_\partial(\mu_x,\mu_y)}
                  \mathrm{e}^{-80((x-\mu_x)^2+(y-\mu_y)^2)},\\*
 u_\mu&=\mu_cu_{\rm sing}+\mu_\alpha A_\mu
                 \mathrm{e}^{\ii(k(x-\mu_x)+\mu_\varphi)},\qquad
 f_\mu=-\Delta u_\mu-k^2u_\mu.
\end{align*}
The reference parameter is $\mu_{\rm ref}=(1,1/2,-1/2,1/2,0)$.
The parameter ranges are $\mu_c\in[0.25,1.75]$,
$\mu_\alpha\in[0.15,1]$, and $\mu_\varphi\in[-\pi,\pi]$.
The packet center $(\mu_x,\mu_y)$ varies within the disk
$(\mu_x+1/2)^2+(\mu_y-1/2)^2\leq0.2^2$, so
$\mu_x\in[-0.7,-0.3]$ and $\mu_y\in[0.3,0.7]$.
During the adaptive trajectory, we use 16 training loads comprising
the reference load and 15 randomized loads, together with 24 independent
test loads to assess reuse. These test loads are excluded from training
and marking in the adaptive runs.

We restrict enrichment to $D=\Omega\cap B_{0.6}(0)$ and use the kernel-lifted
test functions in \eqref{eq:enriched-spaces}. The main ALOD run uses all
16 training loads and ends with rank 2 after 28 coarse cycles.

\begin{figure}[!htbp]
\setlength{\abovecaptionskip}{4pt}
\centering\includegraphics[width=\textwidth]{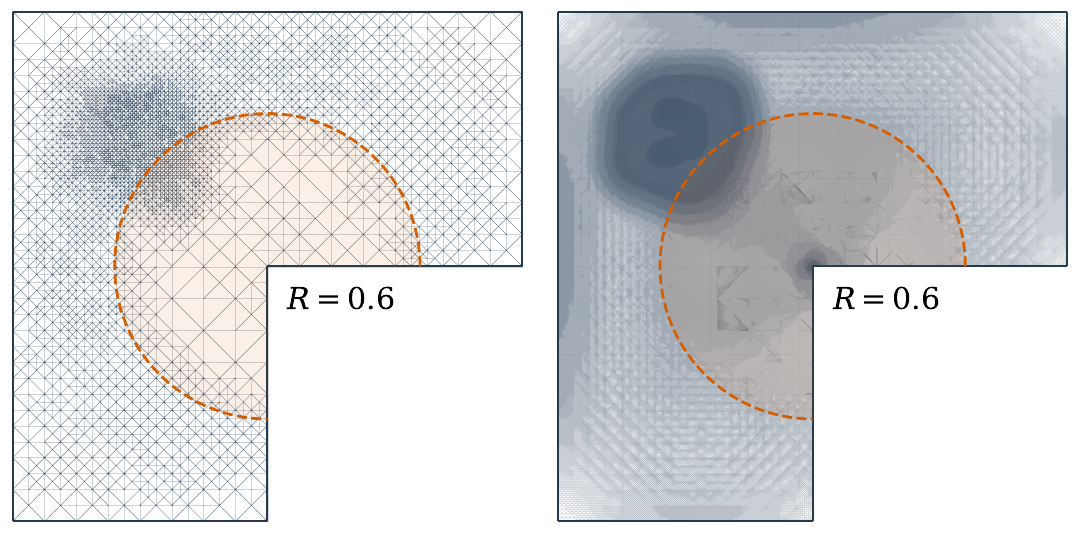}
\caption{Final meshes for the corner problem, from left: ALOD coarse, ALOD fine.
Shading and the dashed arc mark $D=\Omega\cap B_{0.6}(0)$.}
\label{fig:e2-final-meshes}
\end{figure}

Figure~\ref{fig:e2-error-dof} compares the reference-load errors
against the online dof.
The localization check triggers $\ell:2\to3$ at the end of cycle $6$.
At their final displayed points, ALOD, AFEM, UFEM, and LOD yield
$(N_{\rm on},p,E_{\mu_{\rm ref}})$ of
$(4287,0.981,0.003715)$,
$(192763,0.508,0.003297)$,
$(1573888,0.491,0.003523)$, and
$(98560,0.509,0.003564)$, respectively,
where $p$ denotes the observed decay order.
The corresponding tail fits achieve
$R^2=0.9822$, $0.9751$, $0.9991$, and $0.9975$,
using 10 ALOD cycles and 10, 8, and 4 states
for AFEM, UFEM, and LOD, respectively.

\begin{figure}[!htbp]
\setlength{\abovecaptionskip}{4pt}
\centering\includegraphics[width=\textwidth]{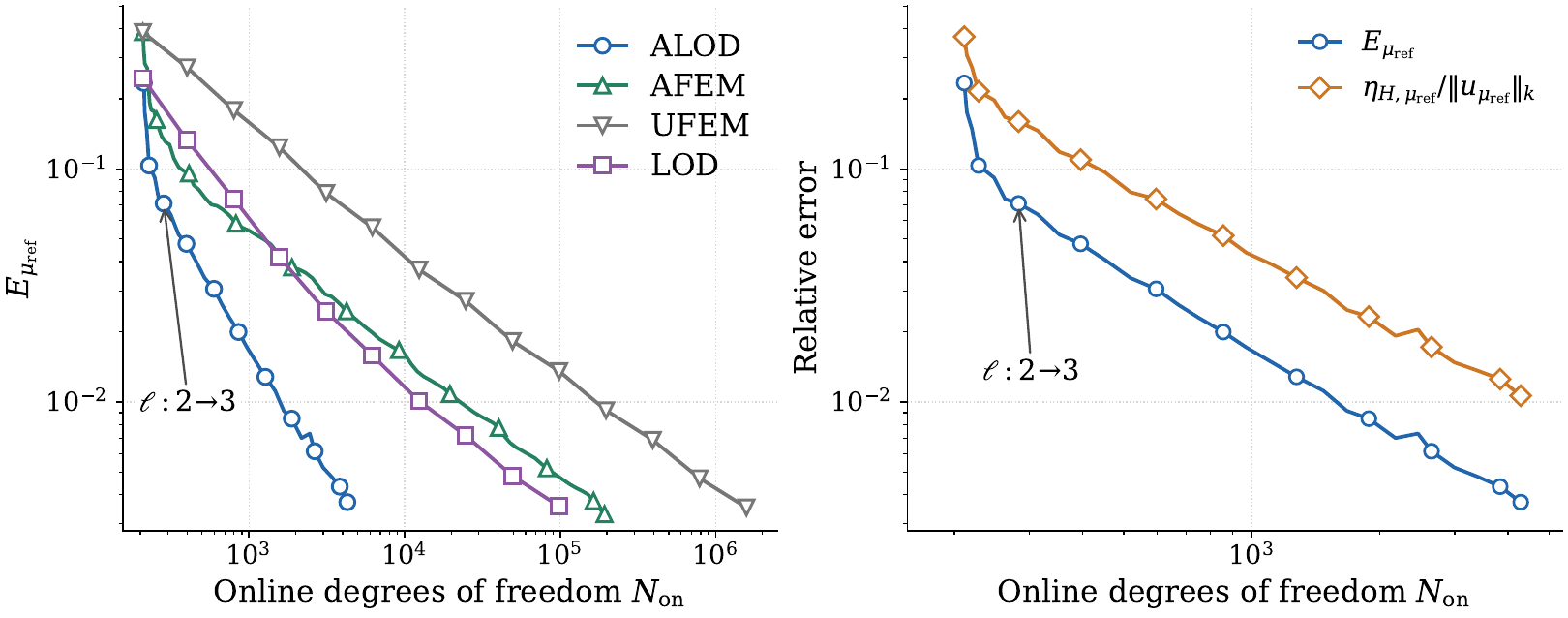}
\caption{Corner-problem convergence for the reference load.
Left: $E_{\mu_{\rm ref}}$ for the four methods. Right: ALOD error and
$\eta_{H,\mu_{\rm ref}}/\norm{u_{\mu_{\rm ref}}}_k$. }
\label{fig:e2-error-dof}
\end{figure}

To examine family-aware construction, we compare independent ALOD
single-load and family-aware trajectories, using only $\mu_{\rm ref}$ and
all 16 training loads, respectively. Both fix $\ell=3$ and run 25 cycles.
Figure~\ref{fig:e2-family-control}, right, shows maximum and median errors
over the same 24 test loads. At $N_{\rm on}=10^3$, logarithmic interpolation
of each test-load curve gives maximum errors $7.809\%$ versus $4.936\%$,
a reduction of $36.8\%$. 

The left panel examines the effects of enrichment and oversampling
on the final mesh pair from the 25-cycle, family-aware run with
fixed $\ell=3$.
We compare the base LOD solution at $\ell=3$ with shared enrichment
retrained from scratch on all 16 training loads at $\ell=2$ and $\ell=3$.
At $\ell=3$, adding two enrichment modes reduces the reference-load
error by $85.85\%$. Comparing the retrained shared spaces at $\ell=2$
and $\ell=3$ gives an $8.28\%$ reduction in the reference-load error.

\begin{figure}[!htbp]
\setlength{\abovecaptionskip}{4pt}
\centering\includegraphics[width=\textwidth]{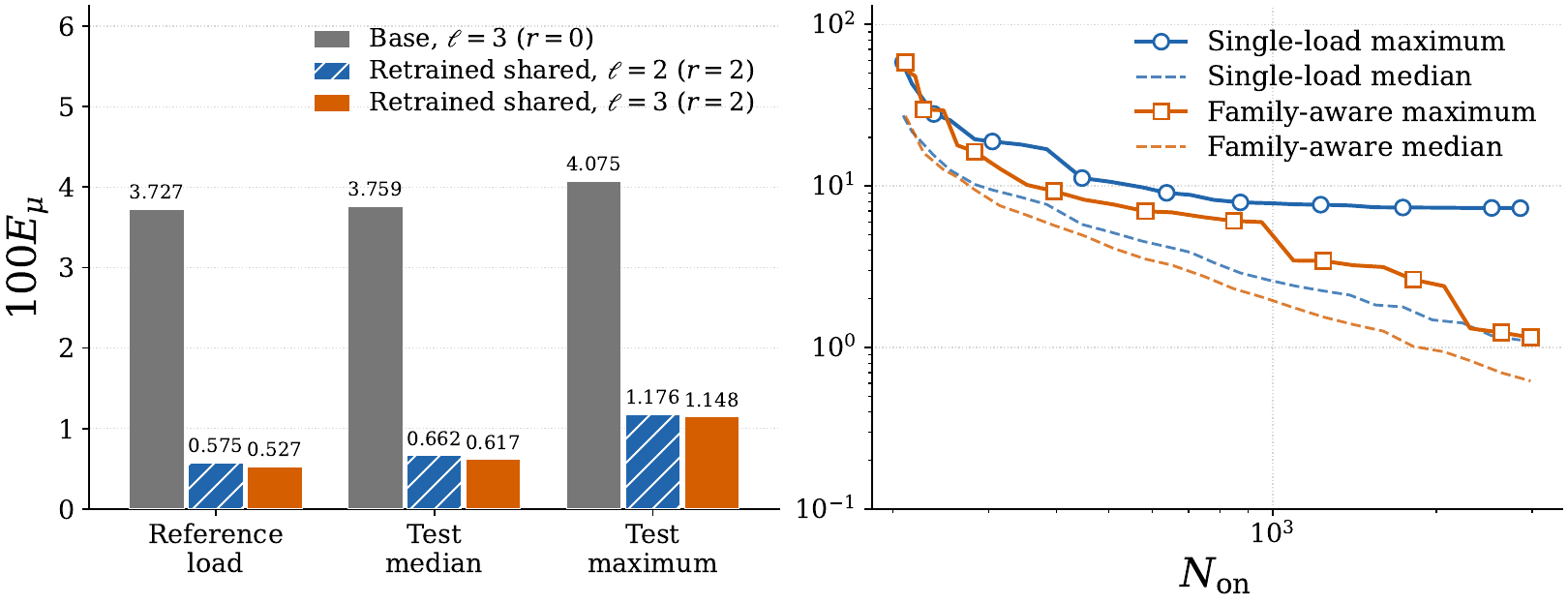}
\caption{Corner-problem controls, with relative energy errors in percent.
Left: fixed-mesh Base ($\ell=3$, $r=0$) versus Retrained shared
($\ell=2,3$, $r=2$), showing reference-load error and test-load median and
maximum. Right: ALOD single-load versus family-aware trajectories at fixed
$\ell=3$; solid curves show the test-load maximum and dashed curves the
median. }
\label{fig:e2-family-control}
\end{figure}

Table~\ref{tab:e2-fixedmesh-control} reports the fixed-mesh reuse results.
Adaptive shared uses the original rank-two
enrichment space from the adaptive family-aware run, unchanged for every
evaluated load. Per-load fresh
instead constructs a new enrichment space from scratch for each evaluated
load on the same mesh pair at $\ell=3$, producing rank one in every case.
The latter uses the evaluated load in its own enrichment construction and
is therefore a per-load diagnostic rather than shared-space reuse.
The maximum test-load errors are $1.1555\%$ and $0.9712\%$, respectively.
Reuse is effective for this prescribed family, with some loss in the
worst-case test-load error.

\begin{table}[!htbp]
\caption{Fixed-mesh corner-problem controls. Errors are in percent;
median and maximum are over 24 test loads. Retrained shared is rebuilt
from all 16 training loads. Adaptive shared reuses the adaptive enrichment;
Per-load fresh is constructed separately for each evaluated load.}
\label{tab:e2-fixedmesh-control}\centering\small
\begin{tabular}{@{}lrrrrr@{}}\toprule
Construction & $\ell$ & $r$ & $100E_{\mu_{\rm ref}}$ & Test median & Test maximum\\\midrule
Base & 3 & 0 & 3.7267 & 3.7594 & 4.0749\\
Retrained shared & 2 & 2 & 0.5749 & 0.6619 & 1.1757\\
Retrained shared & 3 & 2 & 0.5273 & 0.6166 & 1.1478\\
\midrule
Adaptive shared & 3 & 2 & 0.5360 & 0.6208 & 1.1555\\
Per-load fresh & 3 & 1 & 0.5623 & 0.6224 & 0.9712\\
\bottomrule\end{tabular}
\end{table}

\FloatBarrier
\subsection{Wavenumber dependence}\label{sec:num-wavenumber}
We consider $\Omega=(0,1)^2$, with Dirichlet conditions on $y=0,1$,
a Neumann condition on $x=0$, and a Robin condition on $x=1$.
In this experiment, we fix the parameter $\mu_0=(3/4,1/2)$:
\begin{align*}
 g_{\mu_0}(x,y)&=x^2(1-x)^2\sin(\pi y)
 \mathrm{e}^{-80((x-3/4)^2+(y-1/2)^2)},\\
 u_k&=g_{\mu_0}\mathrm{e}^{\ii kx},\qquad
 f_k=\mathrm{e}^{\ii kx}(-\Delta g_{\mu_0}-2\ii k\partial_xg_{\mu_0}),
 \qquad k\in\{8,16,32,64,128\}.
\end{align*}
For each $k$, we evaluate the relative energy error
\[
 E=E_{\mu_0}
 =\frac{\norm{u_k-U_k}_k}{\norm{u_k}_k},
\]
where $U_k$ is the numerical solution produced by the method under
consideration. Each ALOD run uses this single load, with a separate operator and correctors,
without regional enrichment. The initial hierarchy satisfies $kH_0=2\sqrt2$ and
$k^3h_0^2=8$.
All ALOD runs start at $\ell=2$, with $\ell_{\max}=5$, $\tau_q=0.3$, and the threshold $\tau_\Theta(k)=3.2/k$.
Oversampling increases to $\ell=3$ at the endpoints of cycles $14$, $12$, $8$,
and $4$ for $k=8,16,32,64$, respectively. At $k=128$, the initial balance
indicator is $\widetilde q_\ell\approx0.4102$, so promotion occurs before
the first accepted state and reduces it to approximately $0.04873$.
Thus, higher wavenumbers trigger earlier oversampling
adjustments in the present experiments, illustrating
the ability of the adaptive criterion to respond to
changes in the localization balance.

Figure~\ref{fig:kappa-error-dof} compares the relative
energy errors against the normalized online dof
$N_{\rm on}/k^2$, where $k^2$ provides a natural
wavelength-scale reference in two dimensions.
Table~\ref{tab:kappa-target-summary} reports the first
crossings of the $1\%$ error threshold.
Across the five wavenumbers, ALOD requires between
$985$ and $6478$ online dof to reach
the target accuracy, compared with $50288$ to $817056$
for AFEM.
The corresponding AFEM-to-ALOD dimension ratios range
from approximately $51$ to $126$.
LOD also requires substantially larger online spaces
for $k=8,16,32,64$.
At $k=128$, ALOD achieves an error of $0.9701\%$
with $6478$ online degrees of freedom, whereas LOD
remains at $1.1536\%$ with $65535$ degrees of freedom
before reaching the memory limit.

Over the tested range, as $k$ increases from $8$ to $128$, $N_{\rm on}$ grows
from $985$ to $6478$, a factor of approximately $6.58$,
while $k^2$ increases by a factor of $256$.
Consequently, $N_{\rm on}/k^2$ decreases from
approximately $15.39$ to $0.3954$, a reduction by
nearly a factor of $39$.

\begin{figure}[!htbp]
\centering\includegraphics[width=\textwidth]{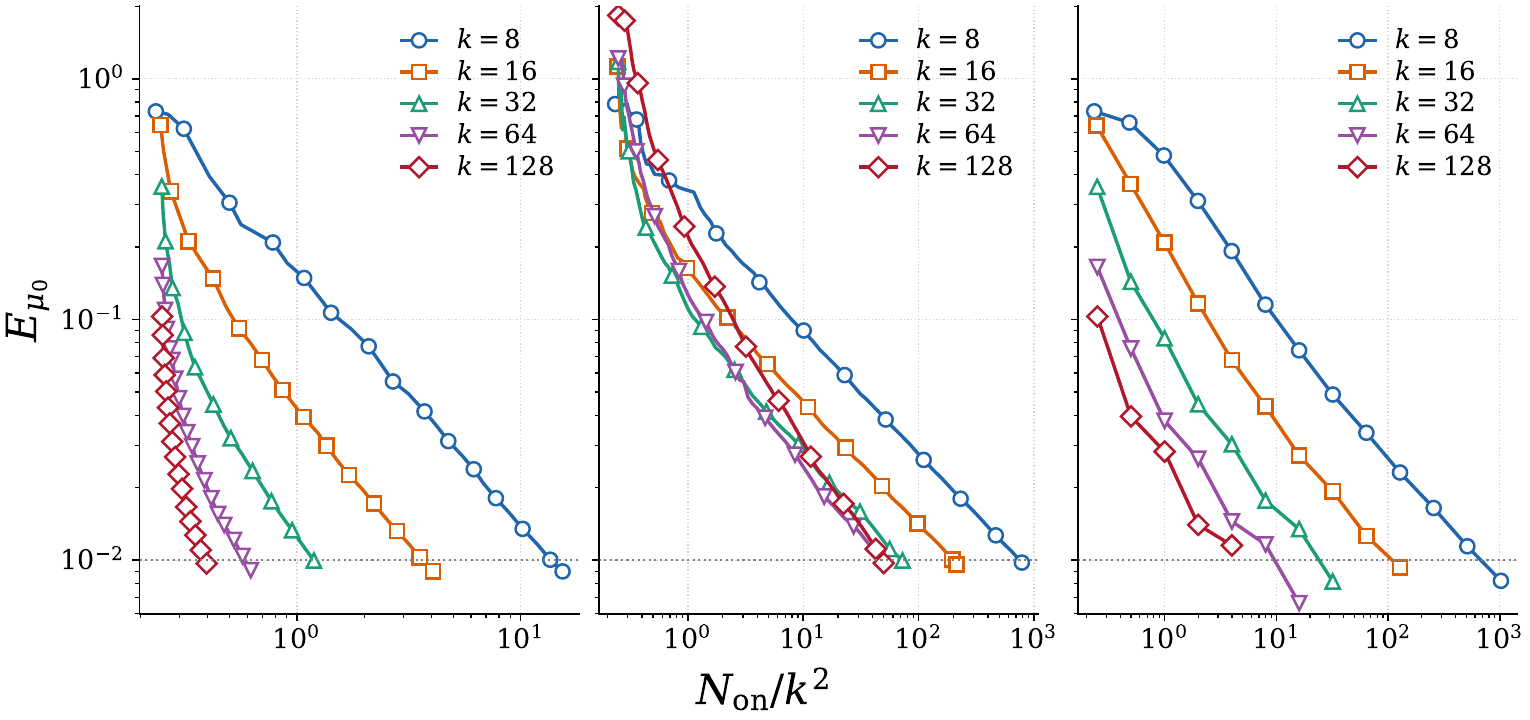}
\caption{Relative energy error versus $N_{\rm on}/k^2$,
with ALOD, AFEM, and LOD from left to right. The dotted line marks $0.01$.
AFEM and LOD are retained historical comparisons for the same loads.}
\label{fig:kappa-error-dof}
\end{figure}

\begin{table}[!htbp]
\caption{First $0.01$ crossings. Each pair gives
online dimension and relative error in percent. The entry marked $\dagger$
is the final LOD state before the memory limit.}
\label{tab:kappa-target-summary}\centering\small
\setlength{\tabcolsep}{5pt}
\begin{tabular}{@{}rrrrrrr@{}}\toprule
&\multicolumn{2}{c}{ALOD}&\multicolumn{2}{c}{AFEM}&\multicolumn{2}{c}{LOD}\\
\cmidrule(lr){2-3}\cmidrule(lr){4-5}\cmidrule(l){6-7}
$k$ & $N_{\rm on}$ & $100E$ & $N_{\rm on}$ & $100E$ & $N_{\rm on}$ & $100E$\\\midrule
8 & 985 & 0.8992 & 50288 & 0.9786 & 65535 & 0.8222\\
16 & 1039 & 0.9007 & 54549 & 0.9595 & 32767 & 0.9333\\
32 & 1220 & 0.9975 & 74239 & 0.9984 & 32767 & 0.8167\\
64 & 2549 & 0.9096 & 197070 & 0.9688 & 65535 & 0.6640\\
128 & 6478 & 0.9701 & 817056 & 0.9731 & 65535 & $1.1536^\dagger$\\
\bottomrule\end{tabular}
\end{table}
\FloatBarrier

\iffalse
\section{Conclusions}\label{sec:conclusions}
ALOD separates the online coarse space from the fine-scale resolution used
for its construction. Local residual Riesz representatives guide adaptivity
and provide residual and corrector-localization control. For load families,
they also generate a shared regional enrichment. Galerkin orthogonality on
the original LOD test space extends the fine-grid error estimate to enriched
solutions. Kernel-lifted tests ensure stability and quasi-optimality of the
enriched problem under the unified oversampling condition, without additional global
fine-grid solves. 

The experiments demonstrate online compression and rank-two reuse for
the prescribed corner-singular load family at $k=16$, and online compression for a
single smooth load over $k=8,16,32,64,128$. The analysis can be extended to the variable-coefficient problem
\begin{equation*}
 -\nabla\cdot(\alpha\nabla u)-k^2nu=f,
\end{equation*}
where $\alpha$ and $n$ are fixed positive scalar-valued coefficients, with
constants depending on their uniform bounds. Related multiscale spectral
methods for heterogeneous Helmholtz problems are studied in
\cite{BrownGallistlPeterseim2017,PeterseimVerfurth2020,
MaAlberScheichl2023}.
Obtaining estimates robust with respect to coefficient contrast, however,
would require coefficient-adapted interpolation and stable-decomposition
arguments and is left for future work. A complementary direction is to determine how
the required enrichment rank grows when the load family contains several
independent singular components.
\fi

\section*{Declarations}
The code used for the numerical experiments reported in this work was written using GPT-5.6 Sol (OpenAI). The authors have reviewed the code. The authors assume responsibility for all content.

\bibliographystyle{siamplain}

\begin{thebibliography}{10}

\bibitem{AbdulleHenning2015}
{\sc A.~Abdulle and P.~Henning}, {\em A reduced basis localized orthogonal
  decomposition}, J. Comput. Phys., 295 (2015), pp.~379--401,
  \url{https://doi.org/10.1016/j.jcp.2015.04.016}.

\bibitem{BabuskaMelenk1997}
{\sc I.~Babu{\v{s}}ka and J.~M. Melenk}, {\em The partition of unity method},
  Internat. J. Numer. Methods Engrg., 40 (1997), pp.~727--758,
  \url{https://doi.org/10.1002/(SICI)1097-0207(19970228)40:4<727::AID-NME86>3.0.CO;2-N}.

\bibitem{BabuskaSauter1997}
{\sc I.~Babu{\v{s}}ka and S.~A. Sauter}, {\em Is the pollution effect of the
  {FEM} avoidable for the {Helmholtz} equation considering high wave numbers?},
  SIAM J. Numer. Anal., 34 (1997), pp.~2392--2423,
  \url{https://doi.org/10.1137/S0036142994269186}.

\bibitem{BespalovHaberlPraetorius2017}
{\sc A.~Bespalov, A.~Haberl, and D.~Praetorius}, {\em Adaptive {FEM} with
  coarse initial mesh guarantees optimal convergence rates for compactly
  perturbed elliptic problems}, Comput. Methods Appl. Mech. Engrg., 317 (2017),
  pp.~318--340, \url{https://doi.org/10.1016/j.cma.2016.12.014}.

\bibitem{BinevEtAl2011}
{\sc P.~Binev, A.~Cohen, W.~Dahmen, R.~DeVore, G.~Petrova, and P.~Wojtaszczyk},
  {\em Convergence rates for greedy algorithms in reduced basis methods}, SIAM
  J. Math. Anal., 43 (2011), pp.~1457--1472,
  \url{https://doi.org/10.1137/100795772}.

\bibitem{BuhrEngwerOhlbergerRave2017}
{\sc A.~Buhr, C.~Engwer, M.~Ohlberger, and S.~Rave}, {\em {ArbiLoMod}, a
  simulation technique designed for arbitrary local modifications}, SIAM J.
  Sci. Comput., 39 (2017), pp.~A1435--A1465,
  \url{https://doi.org/10.1137/15M1054213}.

\bibitem{CamargoRojasVega2025}
{\sc L.~Camargo, S.~Rojas, and P.~Vega}, {\em Minimum-residual a posteriori
  error estimates for {HDG} discretizations of the {Helmholtz} equation},
  Comput. Methods Appl. Mech. Engrg., 441 (2025), p.~117981,
  \url{https://doi.org/10.1016/j.cma.2025.117981}.

\bibitem{ChaumontFreletErnVohralik2021}
{\sc T.~Chaumont-Frelet, A.~Ern, and M.~Vohral{\'i}k}, {\em On the derivation
  of guaranteed and {$p$}-robust a posteriori error estimates for the
  {Helmholtz} equation}, Numer. Math., 148 (2021), pp.~525--573,
  \url{https://doi.org/10.1007/s00211-021-01192-w}.

\bibitem{Demlow2010}
{\sc A.~Demlow}, {\em Convergence of an adaptive finite element method for
  controlling local energy errors}, SIAM J. Numer. Anal., 48 (2010),
  pp.~470--497, \url{https://doi.org/10.1137/080741458}.

\bibitem{Doerfler1996}
{\sc W.~D{\"o}rfler}, {\em A convergent adaptive algorithm for {Poisson}'s
  equation}, SIAM J. Numer. Anal., 33 (1996), pp.~1106--1124,
  \url{https://doi.org/10.1137/0733054}.

\bibitem{DuWu2015}
{\sc Y.~Du and H.~Wu}, {\em Preasymptotic error analysis of higher order {FEM}
  and {CIP-FEM} for {Helmholtz} equation with high wave number}, SIAM J. Numer.
  Anal., 53 (2015), pp.~782--804, \url{https://doi.org/10.1137/140953125}.

\bibitem{DuanWu2023}
{\sc S.~Duan and H.~Wu}, {\em Adaptive {FEM} for {Helmholtz} equation with
  large wavenumber}, J. Sci. Comput., 94 (2023), p.~21,
  \url{https://doi.org/10.1007/s10915-022-02074-5}.

\bibitem{EngwerHenningMalqvistPeterseim2019}
{\sc C.~Engwer, P.~Henning, A.~M{\aa}lqvist, and D.~Peterseim}, {\em Efficient
  implementation of the localized orthogonal decomposition method}, Comput.
  Methods Appl. Mech. Engrg., 350 (2019), pp.~123--153,
  \url{https://doi.org/10.1016/j.cma.2019.02.040}.

\bibitem{ErnGuermond2017}
{\sc A.~Ern and J.-L. Guermond}, {\em Finite element quasi-interpolation and
  best approximation}, ESAIM Math. Model. Numer. Anal., 51 (2017),
  pp.~1367--1385.

\bibitem{EsterhazyMelenk2012}
{\sc S.~Esterhazy and J.~M. Melenk}, {\em On stability of discretizations of
  the {Helmholtz} equation}, in Numerical Analysis of Multiscale Problems,
  vol.~83 of Lecture Notes in Computational Science and Engineering, Springer,
  Heidelberg, 2012, pp.~285--324.

\bibitem{FreeseHauckPeterseim2024}
{\sc P.~Freese, M.~Hauck, and D.~Peterseim}, {\em Super-localized orthogonal
  decomposition for high-frequency {Helmholtz} problems}, SIAM J. Sci. Comput.,
  46 (2024), pp.~A2377--A2397, \url{https://doi.org/10.1137/21M1465950}.

\bibitem{GalkowskiSpence2025}
{\sc J.~Galkowski and E.~A. Spence}, {\em Sharp preasymptotic error bounds for
  the {Helmholtz} {$h$}-{FEM}}, SIAM J. Numer. Anal., 63 (2025), pp.~1--22,
  \url{https://doi.org/10.1137/23M1546178}.

\bibitem{GallistlPeterseim2015}
{\sc D.~Gallistl and D.~Peterseim}, {\em Stable multiscale {Petrov--Galerkin}
  finite element method for high frequency acoustic scattering}, Comput.
  Methods Appl. Mech. Engrg., 295 (2015), pp.~1--17,
  \url{https://doi.org/10.1016/j.cma.2015.06.017}.

\bibitem{GuoLi2026}
{\sc Y.~Guo and Y.~Li}, {\em Localized pointwise a posteriori error estimates
  for nonconforming finite element methods}, 2026,
  \url{https://doi.org/10.48550/arXiv.2608.02132}.
\newblock arXiv:2608.02132.

\bibitem{HauckPeterseim2022}
{\sc M.~Hauck and D.~Peterseim}, {\em Multi-resolution localized orthogonal
  decomposition for {Helmholtz} problems}, Multiscale Model. Simul., 20 (2022),
  pp.~657--684, \url{https://doi.org/10.1137/21M1414607}.

\bibitem{HellmanMalqvist2017}
{\sc F.~Hellman and A.~M{\aa}lqvist}, {\em Contrast independent localization of
  multiscale problems}, Multiscale Model. Simul., 15 (2017), pp.~1325--1355,
  \url{https://doi.org/10.1137/16M1100460}.

\bibitem{HenningMalqvist2014}
{\sc P.~Henning and A.~M{\aa}lqvist}, {\em Localized orthogonal decomposition
  techniques for boundary value problems}, SIAM J. Sci. Comput., 36 (2014),
  pp.~A1609--A1634, \url{https://doi.org/10.1137/130933198}.

\bibitem{KeilRave2023}
{\sc T.~Keil and S.~Rave}, {\em An online efficient two-scale reduced basis
  approach for the localized orthogonal decomposition}, SIAM J. Sci. Comput.,
  45 (2023), pp.~A1491--A1518, \url{https://doi.org/10.1137/21M1460016}.

\bibitem{KornhuberPeterseimYserentant2018}
{\sc R.~Kornhuber, D.~Peterseim, and H.~Yserentant}, {\em An analysis of a
  class of variational multiscale methods based on subspace decomposition},
  Math. Comp., 87 (2018), pp.~2765--2774.

\bibitem{KunischVolkwein2002}
{\sc K.~Kunisch and S.~Volkwein}, {\em {Galerkin} proper orthogonal
  decomposition methods for a general equation in fluid dynamics}, SIAM J.
  Numer. Anal., 40 (2002), pp.~492--515,
  \url{https://doi.org/10.1137/S0036142900382612}.

\bibitem{LafontaineSpenceWunsch2022}
{\sc D.~Lafontaine, E.~A. Spence, and J.~Wunsch}, {\em A sharp relative-error
  bound for the {Helmholtz} {$h$}-{FEM} at high frequency}, Numer. Math., 150
  (2022), pp.~137--178, \url{https://doi.org/10.1007/s00211-021-01253-0}.

\bibitem{LiShui2026}
{\sc Y.~Li and H.~Shui}, {\em Smoother-type a posteriori error estimates for
  finite element methods}, Comput. Methods Appl. Mech. Engrg., 453 (2026),
  p.~118847, \url{https://doi.org/10.1016/j.cma.2026.118847}.

\bibitem{LiZikatanov2021}
{\sc Y.~Li and L.~T. Zikatanov}, {\em A posteriori error estimates of finite
  element methods by preconditioning}, Comput. Math. Appl., 91 (2021),
  pp.~192--201, \url{https://doi.org/10.1016/j.camwa.2020.08.001}.

\bibitem{LiZikatanov2025}
{\sc Y.~Li and L.~T. Zikatanov}, {\em Nodal auxiliary a posteriori error
  estimates}, Math. Comp.,  (2025), \url{https://doi.org/10.1090/mcom/4141}.
\newblock Published online.

\bibitem{LiZikatanovZuo2025}
{\sc Y.~Li, L.~T. Zikatanov, and C.~Zuo}, {\em Reduced {Krylov} basis methods
  for parametric partial differential equations}, SIAM J. Numer. Anal., 63
  (2025), pp.~976--999, \url{https://doi.org/10.1137/24M1661236}.

\bibitem{MalqvistPeterseim2014}
{\sc A.~M{\aa}lqvist and D.~Peterseim}, {\em Localization of elliptic
  multiscale problems}, Math. Comp., 83 (2014), pp.~2583--2603,
  \url{https://doi.org/10.1090/S0025-5718-2014-02868-8}.

\bibitem{MalqvistPeterseim2020}
{\sc A.~M{\aa}lqvist and D.~Peterseim}, {\em Numerical Homogenization by
  Localized Orthogonal Decomposition}, vol.~5 of SIAM Spotlights, Society for
  Industrial and Applied Mathematics, Philadelphia, PA, 2020,
  \url{https://doi.org/10.1137/1.9781611976458}.

\bibitem{MelenkBabuska1996}
{\sc J.~M. Melenk and I.~Babu{\v{s}}ka}, {\em The partition of unity finite
  element method: Basic theory and applications}, Comput. Methods Appl. Mech.
  Engrg., 139 (1996), pp.~289--314,
  \url{https://doi.org/10.1016/S0045-7825(96)01087-0}.

\bibitem{MelenkSauter2011}
{\sc J.~M. Melenk and S.~Sauter}, {\em Wavenumber explicit convergence analysis
  for {Galerkin} discretizations of the {Helmholtz} equation}, SIAM J. Numer.
  Anal., 49 (2011), pp.~1210--1243, \url{https://doi.org/10.1137/090776202}.

\bibitem{OhlbergerSchindler2015}
{\sc M.~Ohlberger and F.~Schindler}, {\em Error control for the localized
  reduced basis multiscale method with adaptive on-line enrichment}, SIAM J.
  Sci. Comput., 37 (2015), pp.~A2865--A2895,
  \url{https://doi.org/10.1137/151003660}.

\bibitem{Peterseim2017}
{\sc D.~Peterseim}, {\em Eliminating the pollution effect in {Helmholtz}
  problems by local subscale correction}, Math. Comp., 86 (2017),
  pp.~1005--1036, \url{https://doi.org/10.1090/mcom/3156}.

\bibitem{PeterseimScheichl2016}
{\sc D.~Peterseim and R.~Scheichl}, {\em Robust numerical upscaling of elliptic
  multiscale problems at high contrast}, Comput. Methods Appl. Math., 16
  (2016), pp.~579--603, \url{https://doi.org/10.1515/cmam-2016-0022}.

\bibitem{PrudhommeEtAl2002}
{\sc C.~Prud'homme, D.~V. Rovas, K.~Veroy, L.~Machiels, Y.~Maday, A.~T. Patera,
  and G.~Turinici}, {\em Reliable real-time solution of parametrized partial
  differential equations: Reduced-basis output bound methods}, J. Fluids Eng.,
  124 (2002), pp.~70--80, \url{https://doi.org/10.1115/1.1448332}.

\bibitem{RozzaHuynhPatera2008}
{\sc G.~Rozza, D.~B.~P. Huynh, and A.~T. Patera}, {\em Reduced basis
  approximation and a posteriori error estimation for affinely parametrized
  elliptic coercive partial differential equations}, Arch. Comput. Methods
  Eng., 15 (2008), pp.~229--275,
  \url{https://doi.org/10.1007/s11831-008-9019-9}.

\bibitem{ScottZhang1990}
{\sc L.~R. Scott and S.~Zhang}, {\em Finite element interpolation of nonsmooth
  functions satisfying boundary conditions}, Math. Comp., 54 (1990),
  pp.~483--493, \url{https://doi.org/10.1090/S0025-5718-1990-1011446-7}.

\bibitem{Wu2014}
{\sc H.~Wu}, {\em Pre-asymptotic error analysis of {CIP-FEM} and {FEM} for the
  {Helmholtz} equation with high wave number. {Part I}: Linear version}, IMA J.
  Numer. Anal., 34 (2014), pp.~1266--1288,
  \url{https://doi.org/10.1093/imanum/drt033}.

\bibitem{Xu1992}
{\sc J.~Xu}, {\em Iterative methods by space decomposition and subspace
  correction}, SIAM Rev., 34 (1992), pp.~581--613,
  \url{https://doi.org/10.1137/1034116}.

\end{thebibliography}

\end{document}